 \documentclass[final,3p,times]{elsarticle}

\usepackage{graphicx}

\usepackage{amssymb}
\usepackage{amsthm}

\usepackage{amsxtra}
\usepackage{multicol}
\usepackage{multirow}
\usepackage{tabularx}
\usepackage{booktabs}
\usepackage{dcolumn}
\usepackage{hyperref}
\usepackage{dsfont}

\newcommand{\ind}{\mathds{1}}
\newcommand{\N}{\mathbb{N}}

\newcommand{\R}{\mathbb{R}}
\newcommand{\Var}{\text{Var}}
\newcommand{\Cov}{\text{Cov}}
\newcommand{\E}{\mathbb{E}}

\newcommand{\Prob}{\mathbb{P}}

\newcommand{\Sd}{\mathbb{S}^{d-1}}

\newdefinition{dfn}{Definition}
\newtheorem{thm}{Theorem}
\newtheorem{lem}[thm]{Lemma}
\newtheorem{cor}[thm]{Corollary}
\newdefinition{remark}{Remark}
\newdefinition{example}{Example}
\newproof{pf}{Proof}
\newproof{pot}{Proof of Theorem \ref{thm2}}

\journal{Journal of Multivariate Analysis}

\begin{document}

\begin{frontmatter}



\title{Extrapolation of stable random fields\tnoteref{grant}}
\tnotetext[grant]{This research was partially supported by the DFG -- RFFI  09--01--91331 grant, also the second author is supported by the Chebyshev Laboratory (Department of Mathematics and Mechanics, St.-Petersburg State University) under RF government grant 11.G34.31.0026}


\author[ulm]{Wolfgang Karcher}
\ead{wolfgang.karcher@uni-ulm.de}

\author[stp]{Elena Shmileva}
\ead{elena.shmileva@gmail.com}

\author[ulm]{Evgeny Spodarev}
\ead{evgeny.spodarev@uni-ulm.de}

\address[ulm]{Ulm University, Institute of Stochastics, Helmholtzstr. 18, 89081 Ulm, Germany}
\address[stp]{St. Petersburg State University, Chebyshev Laboratory, 14th Line 29b, St. Petersburg 199178, Russia}

\begin{abstract}
In this paper, we discuss three extrapolation methods for $\alpha$-stable random fields with $\alpha \in (1,2]$. We justify them, giving proofs of the existence and uniqueness of the solutions for each method and providing sufficient conditions for path continuity. Two methods are based on minimizing the variability of the difference between the predictor and the theoretical value, whereas in the third approach we provide a new method that maximizes the covariation between these two quantities.
\end{abstract}


\end{frontmatter}


\section{Introduction}

In many applications, Gaussian random fields are chosen as a model for the regionalized variables. However, natural phenomena may exhibit heavy tails such that the assumption of Gaussianity is not reasonable any more. In this case, stable random fields may be a more appropriate choice as they still possess the desirable properties of Gaussian random fields, but allow for heavy tails.

In this paper, we are dealing with the problem of extrapolating an $\alpha$-stable random field. Namely, based on a given set of data points $\{x(t_i)\}_{i=1}^n$ at the locations $t_1,\ldots,t_n$, $n \in \N$, which are part of a realization of an $\alpha$-stable random field, we estimate the unknown value $x(t_0)$ at the location $t_0 \notin \{t_1,\ldots,t_n\}$. We present three linear predictors of the form
\begin{equation}
 \widehat{x(t_0)} = \sum\limits_{i=1}^n \lambda_i x(t_i),\ \ t_0\in \R^d. \label{eq:linear_estimator}
\end{equation}
The weights $\lambda_1,\ldots,\lambda_n \in \R$ in (\ref{eq:linear_estimator}) will be chosen in such a way that the predictor $\widehat{x(t_0)}$ has certain desirable properties. Two predictors are based on ideas of previous work in the literature, namely we generalize the existing methods to a wide range of $\alpha$-stable random fields and investigate the properties of existence, uniqueness, exactness, unbiasedness and continuity of the predictors in detail. We also propose a third new extrapolation method which maximizes the covariation dependence measure between the estimated and theoretical value.

A $2$-stable random field is a Gaussian random field. A widely used extrapolation technique for Gaussian random fields is kriging which was invented by Krige in 1951, see \cite{Kri51}. The idea behind kriging is to find a linear predictor which minimizes the $L^2$-distance between the predictor and the theoretical value. We refer to \cite{Wac98} for further details about kriging. Since the second moment of an $\alpha$-stable random variable does not exist when $\alpha<2$, the kriging technique cannot be applied to $\alpha$-stable random fields so that one has to look for alternative extrapolation methods.

There is a series of papers in the literature dealing with the extrapolation problem of particular stable random processes and fields. In \cite{Pai98}, maximum likelihood estimation and conditional simulation of the fractional stable motion is used to determine the optimal value at an unobserved location. Linear prediction of discrete $\alpha$-stable processes based on minimizing the dispersion between the predictor and the theoretical value is discussed in \cite{BC85,BM98,CS84,GMO00,Hil00,Kok96}. Finally, a modified dispersion measure is minimized in \cite{MM09} in order to obtain a linear prediction of discrete $\alpha$-stable processes with integral representation.

The paper is organized as follows. In Section \ref{sec:preliminaries}, we provide some basics about $\alpha$-stable random fields. In Section \ref{sec:methods}, we present the three extrapolation methods for $\alpha \in (1,2]$ (least scale linear predictor, covariation orthogonal linear predictor and maximum of covariation predictor) and analyze their properties. The proofs of the results can be found in the Appendix.

\section{Preliminaries}\label{sec:preliminaries}

In this section, we provide some basic definitions and results for $\alpha$-stable random variables, random vectors and random fields. For a detailed introduction to $\alpha$-stable distributions and processes, we refer to \cite{ST94}.

\begin{dfn}[$\alpha$-stable random variable]
 Let $0 < \alpha \leq 2$, $\sigma \geq 0$, $-1 \leq \beta \leq 1$, and $\mu \in \R$. A random variable $X$ is said to have an \textit{$\alpha$-stable distribution} with parameters $\sigma$, $\beta$, and $\mu$, if its characteristic function $\varphi_X$ is given by
 \begin{eqnarray*}
  \varphi_X(\theta) = \begin{cases}
                      \exp\left\{-\sigma^\alpha |\theta|^\alpha \left(1-i\beta(\text{sign}(\theta))\tan \frac{\pi \alpha}{2}\right) + i \mu\theta\right\}, & \alpha \neq 1, \\
		      \exp\left\{-\sigma |\theta| \left(1+i\beta \frac{2}{\pi}(\text{sign}(\theta))\ln|\theta|\right) + i \mu\theta\right\}, & \alpha = 1.
                     \end{cases}
 \end{eqnarray*}
\end{dfn}
We briefly write $X \sim S_\alpha(\sigma,\beta,\mu)$. The parameters $\sigma$, $\beta$ and $\mu$ determine the scale, the skewness and the shift of the random variable, respectively, whereas the parameter $\alpha$ specifies the tail behavior of the distribution of $X$. We notice that $S_2(\sigma,0,\mu) = \mathcal{N}(\mu,2\sigma^2)$, where $\mathcal{N}(\mu,2\sigma^2)$ denotes the normal distribution with mean $\mu$ and variance $2\sigma^2$.

If $X\sim S_{\alpha}(\sigma,\beta, \mu)$ with $\alpha\in (0,2)$, then
\begin{eqnarray}
&&\E |X|^p<\infty \text{\quad for any\quad} p\in (0,\alpha), \\
&&\E |X|^p=\infty \text{\quad for any\quad} p\in [\alpha,\infty). \label{eq:mean_infinity}
\end{eqnarray}
Moreover if $\alpha \in (0,2)$ and $\beta=0$ in the case $\alpha=1$, then for every $p\in (0,\alpha)$ there exists a constant $c_{\alpha,\beta}(p)>0$ such that
\begin{equation}
\left(\E\vert X \vert^p\right)^{1/p} = c_{\alpha,\beta}(p) \sigma.\label{eq:p_mean}
\end{equation}

An $\alpha$-stable random vector is defined as follows. Let $d \in \N$ and $\langle \cdot,\cdot \rangle$ denote the Euclidean scalar product in  $\R^d$.

\begin{dfn}[$\alpha$-stable random vector]
 Let $0 < \alpha < 2$. We say that a random vector $\boldsymbol{X} = (X_1,\ldots,X_d)^\mathsf{T}$ in $\R^d$ has a \textit{(multivariate) $\alpha$-stable distribution} if there exists a finite measure $\Gamma$ on the unit sphere $\Sd$ of $\R^d$ and a vector $\boldsymbol{\mu}$ in $\R^d$ such that the characteristic function $\varphi_{\boldsymbol{X}}$ is given by
\begin{eqnarray*}
 \varphi_{\boldsymbol{X}}(\boldsymbol{\theta}) = \begin{cases}
                    \exp\left\{-\int_{\Sd} |\langle\boldsymbol{\theta},\boldsymbol{s}\rangle|^\alpha\left(1-i\text{sign}(\langle\boldsymbol{\theta},\boldsymbol{s}\rangle)\tan \frac{\pi\alpha}{2}\right)\Gamma(d\boldsymbol{s}) + i\langle\boldsymbol{\theta},\boldsymbol{\mu}\rangle\right\}, &\alpha \neq 1,\\
		    \exp\left\{-\int_{\Sd} |\langle\boldsymbol{\theta},\boldsymbol{s}\rangle|\left(1+i\frac{2}{\pi}\text{sign}(\langle\boldsymbol{\theta},\boldsymbol{s}\rangle)\ln|\langle\boldsymbol{\theta},\boldsymbol{s}\rangle|\right)\Gamma(d\boldsymbol{s}) + i\langle\boldsymbol{\theta},\boldsymbol{\mu}\rangle\right\}, &\alpha = 1.
                   \end{cases}
\end{eqnarray*}
\end{dfn}
The pair $(\Gamma,\boldsymbol{\mu})$ is unique. The \textit{spectral measure} $\Gamma$ on the unit sphere $\Sd$ contains most information about $\boldsymbol{X}$ such as the dependence structure and distributional properties of its components. The vector $\boldsymbol{\mu}$ determines the shift with respect to the origin. We say that $\Gamma$ is \textit{concentrated on a great sub-sphere} of $\Sd$ if it is concentrated on the intersection of $\Sd$ with a $(d-1)$-dimensional linear subspace. 
If $\Gamma$ is not concentrated on a great sub-sphere of $\Sd$, then $\boldsymbol{X}$ is called \textit{full-dimensional}. Otherwise, $\boldsymbol{X}$ is called \textit{singular}.

\begin{lem}[Linear dependence]\label{lemma:lin_dep}
Let $d \in \N$ and consider a $d$-dimensional $\alpha$-stable random vector $\boldsymbol{X}=(X_1,\ldots,X_d)^\mathsf{T}$. Let $\boldsymbol{\mu}=(0,\ldots,0)^\mathsf{T} \in \R^d$. $\boldsymbol{X}$ is singular if and only if $\sum_{i=1}^d c_i X_i = 0$ almost surely for some $(c_1,\ldots,c_d)^\mathsf{T} \in \R^d \setminus \{0\}$.
\end{lem}

\begin{dfn}[Linear regression]
Let $n \in \N$ and $(X_1,\ldots,X_n,X_{n+1})^\mathsf{T}$ be an $\alpha$-stable random vector. The regression of $X_{n+1}$ on $(X_1,\ldots,X_n)^\mathsf{T}$ is called \textit{linear} if there exist some constants $c_1,\ldots,c_n \in \R$ such that
\begin{equation}
\E(X_{n+1}|X_1,\ldots,X_n) = \sum_{i=1}^n c_i X_i \quad a.s. \label{eq:linear_regression}
\end{equation}
\end{dfn}

Let $\text{sp}(\boldsymbol{X})$ denote the linear span of the $\alpha$-stable random vector $\boldsymbol{X}$.
\begin{dfn}[Multiple regression property]
Let $n \in \N$. The vector $\boldsymbol{X}$ has the \textit{multiple regression property} if for all random variables $Y_1,\ldots,Y_{n+1} \in \text{sp}(\boldsymbol{X})$
$$\E(Y_{n+1}|Y_1,\ldots,Y_{n}) \in \text{sp}\left((Y_1,\ldots,Y_n)^\mathsf{T}\right).$$
\end{dfn}


A random vector $\boldsymbol{X}$ in $\R^d$ is called \textit{symmetric} if $\Prob(\boldsymbol{X} \in A) = \Prob(-\boldsymbol{X} \in A)$ for any Borel set $A\in \R^d$. If $\boldsymbol{X}$ is a symmetric $\alpha$-stable random vector in $\R^d$, its characteristic function $\varphi_{\boldsymbol{X}}$ is given by
$$\varphi_{\boldsymbol{X}}(\boldsymbol{\theta}) =
      \exp\left\{-\int_{\Sd} |\langle\boldsymbol{\theta},\boldsymbol{s}\rangle|^\alpha \Gamma(d\boldsymbol{s})\right\}, \quad \boldsymbol{\theta} \in \R^d.$$
For symmetric $\alpha$-stable distributions, we use the standard abbreviation  $S\alpha S$.


The dependence of two $\alpha$-stable random variables cannot be studied by using the covariance because of the absence of the second moments if $\alpha<2$. Thus, we have to use a different dependence measure. The most popular dependence measure between two $\alpha$-stable random variables is the covariation. It is only defined for $\alpha$-stable random variables with $1 < \alpha \leq 2$. Let $a^{<p>} := \vert a \vert^p \text{sign}(a)$ denote the signed power for $a\in \R$ and $p\in \R$.

\begin{dfn}[Covariation]\label{def:covariation1}
Let $\alpha \in (1,2]$ and $\Gamma$ be the spectral measure of the $\alpha$-stable random vector $\boldsymbol{X}=(X_1,X_2)^\mathsf{T}$. The \textit{covariation} of $X_1$ on $X_2$ is the real number
$$\left[X_1,X_2\right]_\alpha :=  \int_{\mathbb{S}^1} s_1 s_2^{<\alpha-1>} \Gamma(ds_1,ds_2).$$
\end{dfn}
The covariation is linear in the first entry, but in general not in the second. It is clear that $[X_1,X_2]_\alpha = 0$ if $X_1$ and $X_2$ are independent. For $\alpha=2$, we have $\left[X_1,X_2\right]_2 = 1/2 \cdot \Cov(X_1,X_2)$, where $\Cov(X_1,X_2)$ denotes the covariance between $X_1$ and $X_2$, see~\cite[Example~2.7.2]{ST94}.


Let $(\Omega,\mathcal{F},\Prob)$ be the underlying probability space and $L_0(\Omega)$ be the set of all real random variables defined on it. Furthermore, let $(E,\mathcal{E},m)$ be an arbitrary measurable space, $\beta: E \to [-1,1]$ be a measurable function and $\mathcal{E}_0 := \{A \in \mathcal{E}: m(A) < \infty\}$.
\begin{dfn}[$\alpha$-stable random measure]
An independently scattered $\sigma$-additive set function $M: \mathcal{E}_0 \to L_0(\Omega)$ such that for each $A \in \mathcal{E}_0$
 $$M(A) \sim S_\alpha\left((m(A))^{1/\alpha}, \frac{\int_A\beta(x)m(dx)}{m(A)}, 0\right)$$
 is called \textit{$\alpha$-stable random measure} on $(E, \mathcal{E})$ with control measure $m$ and skewness  $\beta$.
\end{dfn}
\textit{Independent scatteredness} means that for any collection $A_1,A_2, \ldots, A_n$, $n \in \N$, of disjoint sets belonging to $\mathcal{E}_0$, the random variables $M(A_1),\ldots,M(A_n)$ are independent, whereas \textit{$\sigma$-additivity} means that if $\bigcup_{j=1}^\infty A_j \in \mathcal{E}_0$ for a sequence of disjoint sets $A_1,A_2,\ldots \in \mathcal{E}_0$, then $M(\bigcup_{j=1}^\infty A_j) = \sum_{j=1}^\infty M(A_j)$ almost surely.

A collection $X=\{X(t), t \in \R^d\}$ of $\alpha$-stable random variables is called \textit{$\alpha$-stable random field}. We consider random fields of the form
\begin{equation}
 X(t) = \int_E f_t(x) M(dx), \quad t \in \R^d, \quad d \in \N, \label{eq:integral_representation}
\end{equation}
where $f_t:E \to \R$ are measurable functions such that $\int_E \vert f_t(x) \vert^\alpha m(dx) < \infty$ and in the case $\alpha=1$ additionally $\int_E \vert f(x) \beta(x) \ln \vert f(x)\vert \vert m(dx) < \infty$. The stochastic integral in (\ref{eq:integral_representation}) is defined in the natural way by approximating the functions $f_t$ by simple functions. We notice that the random field $X$ in (\ref{eq:integral_representation}) is $\alpha$-stable because its finite-dimensional distributions are $\alpha$-stable. In particular, the scale parameter $\sigma_{X(t)}$ of $X(t)$ is given by
\begin{equation}
\sigma_{X(t)}^\alpha = \int_E \vert f_t(x) \vert^\alpha m(dx). \label{eq:integral_scale}
\end{equation}
If $\beta(x)=0$ for all $x \in E$, then the corresponding  $\alpha$-stable random measure has $S\alpha S$ values and is called \textit{$S\alpha S$ measure}. In this case, the finite-dimensional distributions of $X$ are symmetric $\alpha$-stable, so we have that $X$ is an $S\alpha S$ random field.

Random fields with integral representation (\ref{eq:integral_representation}) comprise a rich subclass of $\alpha$-stable random fields, namely random fields which are separable in probability, see \cite[Chapter~13]{ST94}. In particular, all stochastically continuous $\alpha$-stable random fields are separable in probability.

Let $X_j = \int_E f_j(x) M(dx)$, $j=1,2,\ldots$, and $X=\int_E f(x) M(dx)$. Then for $\alpha \neq 1$, it can be shown that
\begin{equation}
\underset{j \to \infty}{\text{plim}} X_j = X \quad \Leftrightarrow \quad
\lim\limits_{j \to \infty} \sigma_{X_j-X} = 0, \label{eq:stochastic_convergence}
\end{equation}
where ${\text{plim}}$ denotes convergence in probability, cf.~\cite[Proposition 3.5.1]{ST94}.

In \cite[Theorem~1 and Remark~1]{CW92}, it is shown that for $t_0,t_1,\ldots,t_n \in \R^d$ it holds
$$\E(X(t_0)|X(t_1),\ldots,X(t_n)) = \sum_{i=1}^n c_i X(t_i)$$
if and only if for all $t \in \R^{n}$
\begin{eqnarray}
 \int_{\mathbb{S}^n} \left(v - \langle c,u\rangle\right)\langle t,u\rangle^{<\alpha-1>} \Gamma(dv,du) = 0 \label{eq:condition1},
\end{eqnarray}
where $\Gamma$ is the spectral measure of the random vector $(X(t_0), X(t_1),\ldots,X(t_n))^\mathsf{T}$. If $X$ is a moving average process, another criterion based on kernel functions for the regression of $X(t_0)$ on $(X(t_1),\ldots,X(t_n))^\mathsf{T}$ to be linear is given in \cite[Section~4.4.]{CW92}.

\begin{lem}\label{lemma:great_subsphere}
Consider a $d$-dimensional $\alpha$-stable random vector $\boldsymbol{X}=(X_1,\ldots,X_d)^\mathsf{T}$ with integral representation
$$\boldsymbol{X} = \left(\int_E f_1(x) M(dx),\ldots,\int_E f_d(x) M(dx)\right)^\mathsf{T}.$$
Then $\boldsymbol{X}$ is singular if and only if $\sum_{i=1}^d c_i f_i(x) = 0$ $m$-almost everywhere for some vector $(c_1,\ldots,c_d)^\mathsf{T} \in \R^d \setminus \{0\}$.
\end{lem}
Lemma~\ref{lemma:great_subsphere} follows from Lemma~\ref{lemma:lin_dep} and the fact that $\sigma_{\sum_{i=1}^d c_i X_i}^\alpha = \int_E \left\vert \sum_{i=1}^d c_i f_i(x) \right\vert^\alpha m(dx)$.

\begin{dfn}[Covariation function]
Let $1<\alpha\leq 2$ and $X$ be an $\alpha$-stable random field. The function $\kappa:\R^d\times \R^d \to \R$ defined by
 $$\kappa(s,t) = [X(s),X(t)]_\alpha, \quad s,t \in \R^d,$$
 is called the \textit{covariation function} of $X$.
\end{dfn}
For any $\alpha$-stable random field $X$ of the form (\ref{eq:integral_representation}) with $1<\alpha\leq 2$, its covariation function is given by 
 \begin{equation}
 \kappa(s,t) = \int_E f_{s}(x) f_{t}(x)^{<\alpha-1>} m(dx). \label{eq:covariation_kernel}
 \end{equation}

A random vector $(X,Y)^\mathsf{T}$ is called \textit{associated} if, for any functions $f,g:\R^2 \to \R$ which are non-decreasing in each argument, one has
\begin{equation}
\Cov(f(X,Y),g(X,Y)) \geq 0
\end{equation}
whenever the covariance exists. It is called \textit{negatively associated} if for any functions $f,g:\R \to \R$ which are non-decreasing, one has
\begin{equation*}
\Cov(f(X),g(Y)) \leq 0
\end{equation*}
whenever the covariance exists. Notice that the condition $\Cov(f(X,Y),g(X,Y)) \leq 0$ does not make sense because $\Cov(g(X,Y),g(X,Y)) \geq 0$ for any non-degenerate random vector $(X,Y)^\mathsf{T}$.

\begin{lem}[Association, kernel functions]\label{lemma:association_kernels}
Let $0<\alpha<2$ and $(X,Y)^\mathsf{T}$ be an $\alpha$-stable random vector with integral representation $(X,Y)^\mathsf{T} = \left(\int_E f_1(x)M(dx),\int_E f_2(x)M(dx)\right)^\mathsf{T}$. Then $(X,Y)^\mathsf{T}$ is associated (negatively associated) if and only if $f_1f_2\geq 0$ ($f_1f_2\leq 0$) $m$-almost everywhere.
\end{lem}

\begin{cor}[Decomposition of a stable random vector]\label{corollary:decomposition}
Let $0<\alpha<2$ and $(X,Y)^\mathsf{T}$ be an $\alpha$-stable random vector. Then there exist $\alpha$-stable random variables $X_1,X_2,Y_1,Y_2$ such that
\begin{equation*}
X = X_1+X_2 \quad \text{and} \quad Y=Y_1+Y_2 \quad \text{a.s.,}
\end{equation*}
$(X_1,Y_1)^\mathsf{T}$ is associated, $(X_2,Y_2)^\mathsf{T}$ is negatively associated and the components of each of the random vectors $(X_1,X_2)^\mathsf{T}$, $(X_1,Y_2)^\mathsf{T}$, $(Y_1,X_2)^\mathsf{T}$, $(Y_1,Y_2)^\mathsf{T}$ are independent.
\end{cor}

Due to Corollary~\ref{corollary:decomposition}, the dependence relation between two $\alpha$-stable random variables can be reduced to the association and negative association relations between their parts. The following corollary provides a justification for using the covariation as a dependence measure. It is a direct consequence of Lemma~\ref{lemma:association_kernels} and formula (\ref{eq:covariation_kernel}).

\begin{cor}[Association, covariation]\label{corollary:association}
Let $1<\alpha<2$ and $(X,Y)^\mathsf{T}$ be an $\alpha$-stable random vector. If $(X,Y)^\mathsf{T}$ is associated (negatively associated), then $[X,Y]_\alpha \geq 0$ ($[X,Y]_\alpha \leq 0$).
\end{cor}

The covariation is connected to the mixed moments by the following lemma.

\begin{lem}\label{lemma:mixed_moments}
Let $1<\alpha<2$ and suppose that $(X,Y)^\mathsf{T}$ is an $\alpha$-stable random vector with spectral measure $\Gamma$ such that $X \sim S_\alpha(\sigma_X, \beta_X,0)$ and $Y \sim S_\alpha(\sigma_Y,\beta_Y,0)$. For $1\leq p<\alpha$, it holds
\begin{equation}
\frac{\E\left(X Y^{<p-1>}\right)}{\E |Y|^p}=\frac{[X,Y]_{\alpha}(1-c\cdot \beta_Y)+ c\cdot (X,Y)_{\alpha}}{\sigma_Y^{\alpha}}, \label{eq:mixed_moments}
\end{equation}
where $(X,Y)_\alpha:= \int_{\mathbb{S}^1} s_1\vert s_2\vert^{\alpha-1}\Gamma(ds)$ and
\begin{equation*}
c:=c_{\alpha, p}(\beta_Y) = \frac{\tan (\alpha \pi/2)}{ 1+\beta_Y^2  \tan^2 (\alpha \pi/2 )} \left[ \beta_Y \tan(\alpha \pi/2)-\tan \left(\frac{p}{\alpha} \arctan(\beta_Y  \tan (\alpha \pi/2))\right)\right].
\end{equation*}
If $Y$ is symmetric, i.~e. $\beta_Y=0$, then $c=0$.
\end{lem}

A random vector $\boldsymbol{X}$ is called \textit{sub-Gaussian} if $\boldsymbol{X}\stackrel{d}{=}A^{1/2} \boldsymbol{G}$, where $A \sim S_{\alpha/2}((\cos(\pi\alpha/4))^{2/\alpha},1,0)$ and $\boldsymbol{G}$ is a zero mean Gaussian vector independent of $A$. This multivariate distribution is a particular case of the Gaussian scale mixture and the elliptical distributions. Consider a stationary zero mean Gaussian random field $G=\{G(t), t \in \R^d\}$ with a positive definite covariance function $C$. Assume that the random variable $A$ is independent of $G$. The random field
$$X\stackrel{d}{=}\{A^{1/2}G(t), t \in \R^d\}$$
is called \textit{sub-Gaussian random field}. In \cite[Example~2.7.4]{ST94}, it is shown that for sub-Gaussian random fields, the covariation function is given by
\begin{equation}
\kappa(s,t) = 2^{-\alpha/2} C(s-t)C(0)^{(\alpha-2)/2}. \label{eq:covariation_function_sub_Gaussian}
\end{equation}

\section{Extrapolation methods}\label{sec:methods}

Let $X$ be an $\alpha$-stable random field of the form (\ref{eq:integral_representation}) and $1<\alpha\leq 2$. In Sections \ref{subsec:LSL}--\ref{subsec:MCL}, we present three approaches that yield the weights $\lambda_1,\ldots,\lambda_n \in \R$ of the linear predictor
\begin{equation}
\widehat{X(t_0)} = \sum_{i=1}^n \lambda_i X(t_i) \label{eq:linear_estimator2}
\end{equation}
for the value $X(t_0)$ at the location $t_0 \in \R^d$ based on the values $X(t_1),\ldots,X(t_n)$ at the neighboring locations $t_1,\ldots,t_n$, $n \in \N$.

A predictor $\widehat{X(t_0)}$ for $X(t_0)$ is called \textit{exact} if $\widehat{X(t_0)} = X(t_0)$ almost surely whenever $t_0 = t_i$ for some $i \in \{1,\ldots,n\}$. We say that $\widehat{X(t_0)}$ is \textit{unbiased} if $\E (\widehat{X(t_0)} - X(t_0)) = 0$. It follows from representation (\ref{eq:integral_representation}) that $\E X(t) \equiv 0$ for each $t \in \R^d$, hence the predictor (\ref{eq:linear_estimator2}) is always unbiased.

For each $t_0 \in \R^d$, the linear prediction yields a set of weights $\lambda_1,\ldots,\lambda_n$. Let us consider these weights as functions $\lambda_i:\R^d \to \R$ of the locations $t \in \R^d$ and let $(\Omega,\mathcal{F},\Prob)$ be the underlying probability space of the random field. If the functions $\lambda_i$, $i=1,\ldots,n$, are continuous, the extrapolator $\hat{X}:\R^d\times\Omega \to \R$ given by
\begin{equation*}
\hat{X}(t,\omega) = \sum_{i=1}^n \lambda_i(t) X(t_i,\omega), \quad t \in \R^d,
\end{equation*}
is continuous for each $\omega \in \Omega$. In this case, we say that the predictor is \textit{continuous}.


\subsection{Least Scale Linear (LSL) predictor}\label{subsec:LSL}

We noticed that the idea behind kriging is the minimization of the prediction variance. Since the variance of a Gaussian random variable is twice the scale parameter of a 2-stable random variable, an obvious generalization of the kriging technique is the minimization of the prediction scale
\begin{equation}
\sigma_{\widehat{X(t_0)}-X(t_0)}^\alpha = \int_E \left|f_{t_0}(x) - \sum_{i=1}^n \lambda_i f_{t_i}(x)\right|^\alpha m(dx) \, \to \, \min \label{eq:LSL}
\end{equation}
with respect to $\lambda_1,\ldots,\lambda_n$. Here, $\sigma_{\widehat{X(t_0)}-X(t_0)}$ is the scale parameter of the $\alpha$-stable random variable $\widehat{X(t_0)}-X(t_0)$, cf.~(\ref{eq:integral_scale}). The predictor $\widehat{X(t_0)}$ based on a solution of this minimization problem is called \textit{Least Scale Linear (LSL) predictor}.

The LSL predictor can be further justified as follows. Let $\{\widehat{X_m(t_0)}\}_{m \in \N}$ be a sequence of predictors for $X(t_0)$. First by (\ref{eq:stochastic_convergence}), $\widehat{X_m(t_0)}$ tends to $X(t_0)$ in probability as $m \to \infty$ if and only if $\sigma_{\widehat{X_m(t_0)}-X(t_0)}$ tends to zero as $m \to \infty$. Second by (\ref{eq:p_mean}), the mean $p$-error between $\widehat{X_m(t_0)}$ and $X(t_0)$ tends to zero for every $p \in (0,\alpha)$ as $m \to \infty$ if and only if $\sigma_{\widehat{X_m(t_0)}-X(t_0)}$ tends to zero as $m \to \infty$. That is why it is reasonable to minimize $\sigma_{\widehat{X(t_0)}-X(t_0)}^\alpha$ in (\ref{eq:LSL}).

We now derive the necessary conditions for the existence of the LSL predictor.

\begin{lem}\label{lemma:first_order_conditions}
Let $M$ be an $\alpha$-stable random measure. Then
 $$\frac{\partial \sigma_{\widehat{X(t_0)}-X(t_0)}^\alpha}{\partial \lambda_j} = \alpha \left[X(t_j),\sum_{i=1}^n\lambda_i X(t_i) - X(t_0) \right]_\alpha, \quad j=1,\ldots,n.$$
\end{lem}

Thus, a solution of the minimization problem (\ref{eq:LSL}) resolves the following system of equations
\begin{equation}
\left[X(t_j),X(t_0) - \sum_{i=1}^n\lambda_i X(t_i)\right]_\alpha=0, \quad j=1, \ldots n, \label{eq:soc_LSL_covariation}
\end{equation}
which can be written as
\begin{equation}
\int_E f_{t_j}(x) \left(f_{t_0}(x) - \sum_{i=1}^n \lambda_i f_{t_i}(x)\right)^{<\alpha-1>}\hspace*{-0.5cm}m(dx)=0, \quad j=1, \ldots n,\label{LSL}
\end{equation}
by (\ref{eq:covariation_kernel}) and the linearity of $\alpha$-stable integrals. Notice that (\ref{LSL}) is nonlinear in $\lambda_1,\ldots,\lambda_n$ if $\alpha < 2$. Therefore, numerical methods have to be applied in this case in order to solve this minimization problem, for example the Levenberg-Marquardt algorithm, see \cite{Mor78}.

\subsubsection{Properties of the LSL predictor}

It is clear that the set of all LSL predictors is closed and convex, see also \cite[p.~59]{DeVL93}.

\begin{thm}[]\label{th:LSLP1}
The LSL predictor has the following properties:
\begin{itemize}
\item[1.] The LSL predictor exists, is unique and exact.
\item[2.] If the random field $X$ is stochastically continuous and the random vector $(X(t_1),\ldots,X(t_n))^\mathsf{T}$ is full-dimensional, then the LSL predictor is continuous.
\end{itemize}
\end{thm}

\begin{remark}
The first and second assertion of Theorem~\ref{th:LSLP1} follow from the properties of the best approximation in $L^\alpha(E,m)$-spaces for $1\leq \alpha\leq 2$. Notice that if $0<\alpha<1$, the quantity $\Vert f \Vert_{L^\alpha(E,m)} = (\int_E \vert f(x)\vert^\alpha m(dx))^{1/\alpha}$ is not a norm anymore (there is neither the triangle inequality nor Minkowski's inequality). Thus for $0<\alpha<1$, the existence and uniqueness of the LSL estimator requires additional investigation.
\end{remark}


\subsection{Covariation Orthogonal Linear (COL) predictor}

Let $X$ be an $\alpha$-stable random field. We have seen that the LSL predictor can be only calculated by numerical methods because the system of equations (\ref{LSL}) is nonlinear. The linearity of the covariation in the first entry suggests to exchange both entries in (\ref{eq:soc_LSL_covariation}) such that we get the system of linear equations
\begin{equation}
\left[X(t_0),X(t_j)\right]_\alpha - \sum_{i=1}^n \lambda_i\left[X(t_i),X(t_j)\right]_{\alpha} = 0, \quad j=1,\ldots,n. \label{eq:soc2}
\end{equation} 
Since the covariation of $X(t_0) - \sum_{i=1}^n \lambda_iX(t_i)$ on $X(t_j)$ is zero for $j=1,\ldots,n$, the linear predictor $\widehat{X(t_0)}=\sum_{i=1}^n \lambda_iX(t_i)$ based on any solution $(\lambda_1,\ldots,\lambda_n)$ of (\ref{eq:soc2}) is called after~\cite{Hil00} \textit{Covariation Orthogonal Linear (COL) predictor}.

The unknown quantities $[X(t_0),X(t_j)]_\alpha$ and $[X(t_i),X(t_j)]_\alpha$ in (\ref{eq:soc2}) can be estimated either by using (\ref{eq:mixed_moments}) in the symmetric case or by estimating the kernel functions and using (\ref{eq:covariation_kernel}) in the non-symmetric case. By (\ref{eq:soc2}), the COL predictor is the solution of the following system of equations
\begin{eqnarray*}
&&[X(t_0),X(t_i)]_\alpha = [\widehat{X(t_0)},X(t_i)]_\alpha, \quad i=1,\ldots,n,
\end{eqnarray*}
which is equivalent to
\begin{eqnarray*}
\E \left(X(t_0) X(t_i)^{<p-1>}\right) = \E \left(\widehat{X(t_0)}X(t_i)^{<p-1>}\right), \quad i=1,\ldots,n,
\end{eqnarray*}
in the case of an $S\alpha S$ random field by (\ref{eq:mixed_moments}).

If the similarity of two $\alpha$-stable random variables is measured by the covariation, the first system of equations ensures that the similarity of the theoretical value $X(t_0)$ and $X(t_i)$ is exactly equal to the similarity of the predictor $\widehat{X(t_0)}$ and $X(t_i)$. The second system of equations guarantees that the ''mixed'' moments between the theoretical value $X(t_0)$ and $X(t_i)$ agree with the ''mixed'' moments between the predictor $\widehat{X(t_0)}$ and $X(t_i)$.

For a detailed treatment of the COL predictor of discrete stationary ARMA $\alpha$-stable processes, we refer to \cite{Hil00}. The following lemma connects linear regression with the COL predictor, see~\cite[Corollary~4.1.3]{ST94}.

\begin{lem}\label{lemma:COL_LSL}
If the regression of $X(t_0)$ on the random vector $(X(t_1),\ldots,X(t_n))^\mathsf{T}$ is linear, i.~e. 
there exists some $(\lambda_1, ..., \lambda_n) \in \R^n$ such that 
$
\E(X(t_0)|X(t_1),\ldots,X(t_n)) = \sum_{i=1}^n \lambda_i X(t_i) 
$
almost surely, then the vector $(\lambda_1,\ldots,\lambda_n)^\mathsf{T}$ is a solution of the system of equations (\ref{eq:soc2}).
\end{lem}

The regression of $X(t_0)$ on $X(t_1)$ is always linear, see~\cite[Theorem~4.1.2]{ST94}. If $n\geq 2$, then $(X(t_0), X(t_1),\ldots,X(t_n))^\mathsf{T}$ has the multiple regression property if and only if $(X(t_0), X(t_1),\ldots,X(t_n))^\mathsf{T}$ is Gaussian or sub-Gaussian, see~\cite[Proposition~4.1.7]{ST94}. Thus for $n \geq 2$, Lemma~\ref{lemma:COL_LSL} can be applied to Gaussian and sub-Gaussian random fields.

In the next subsections, we analyze existence, uniqueness and continuity of the COL predictor and concentrate on moving average, Gaussian and sub-Gaussian random fields. Only for these kinds of $\alpha$-stable random fields, we know how to derive reasonable sufficient conditions such that the properties hold. Notice that for $0<\alpha<2$, any stationary $S\alpha S$ random field $X$ can be decomposed into three independent parts $X=X^{(1)}+X^{(2)}+X^{(3)}$, where $X^{(1)}$ is a mixed moving average, $X^{(2)}$ is harmonizable, and $X^{(3)}$ is a stationary $S\alpha S$ process without mixed moving average or harmonizable components, see \cite{Ros00}. Moving averages are a subclass of mixed moving averages, while sub-Gaussian random fields belong to the classes $X^{(2)}$ or $X^{(3)}$, cf.~\cite[Theorem~4.7.5 and Theorem~6.6.5]{ST94}.

Exactness of the COL predictor holds for all kinds of $\alpha$-stable random fields which directly follows from the corresponding system of equations (\ref{eq:soc2}).

\subsubsection{Properties of the COL predictor for Gaussian and sub-Gaussian random fields}\label{sec:COL_properties_GSG}

Let us first consider a stationary zero mean Gaussian random field $G$ with a positive definite covariance function $C$. The system of linear equations (\ref{eq:soc2}) becomes

\begin{equation}
 \begin{pmatrix}
 C(0) & \cdots & C(t_n-t_1) \\
 \vdots & \ddots & \vdots \\
 C(t_n-t_1) & \cdots & C(0)\end{pmatrix} \begin{pmatrix} 
		\lambda_1 \\
		\vdots \\
		\lambda_n
	      \end{pmatrix} = \begin{pmatrix}
				C(t_0-t_1) \\
				\vdots \\
				C(t_0-t_n)
\end{pmatrix}.\label{eq:soc3}
\end{equation}
Therefore, the matrix on the left hand side of (\ref{eq:soc3}) is positive definite and the system of equations has a unique solution, i.~e. the COL predictor exists and is unique. Notice that the system (\ref{eq:soc3}) coincides with the system of linear equations for the weights in simple kriging, see e.~g. \cite[p.~24]{Wac98}.

In the sub-Gaussian case, the system of linear equations (\ref{eq:soc4}) for the weights of the COL predictor coincides with (\ref{eq:soc3}) due to (\ref{eq:covariation_function_sub_Gaussian}).

\begin{remark}
The fact that this method is well tractable makes it attractive for applications. Nevertheless it cannot cover all $\alpha$-stable random fields because sub-Gaussian random fields have the property of long dependence (see~\cite[Proposition~4.7.4]{ST94}), so they are a very particular case of stable fields.
\end{remark}

\begin{thm}\label{th:COL_continuity}
If the covariance function $C$ is continuous, then the COL predictor for Gaussian and sub-Gaussian random fields is continuous.
\end{thm}

If $t_0 \neq t_i$ for $i=1,\ldots,n$, the Maximum-Likelihood (ML) predictor $\widehat{X(t_0)}$ of $X(t_0)$ is the solution of the optimization problem
$$f_{X(t_0)|X(t_1)=x_1,\ldots,X(t_n)=x_n}(x) \, \to \, \max\limits_{x \in \R}.$$
If $t_0 = t_i$ for some $i=1,\ldots,n$, we have
\begin{equation}
\Prob(X(t_0) = x_i|X(t_1)=x_1,\ldots,X(t_i)=x_i,\ldots,X(t_n)=x_n) = 1\label{eq:ML_exact}
\end{equation}
such that the ML predictor is $\widehat{X(t_0)} = X(t_i)$. Closed formulas for the ML predictor for sub-Gaussian random fields are provided by \cite{Pai98}.

\begin{thm}\label{th:ML_LSL_COL}
For Gaussian and sub-Gaussian random fields, the ML predictor is equal to the COL and LSL predictor.
\end{thm}

There exists an example for which the COL predictor is distinct from the LSL predictor.

\begin{example}\label{ex:LSL_neq_COL}
Let $X$ be the $S\alpha S$ L\'evy motion defined by $X(t)=\int_0^\infty \ind(x\leq t) M(dx)$, where $M$ is an $S\alpha S$ random measure with Lebesgue control measure. Let $t_0 = 3/4$ and $t_1 = 1$. Then the optimization problem for the LSL predictor is
\begin{eqnarray*}
\sigma_{\widehat{X(t_0)}-X(t_0)}^\alpha &=& \int_0^{3/4} \vert 1-\lambda_1 \vert^\alpha dx + \int_{3/4}^1 \vert \lambda_1 \vert^\alpha dx \\
&=& \frac{3}{4}\vert 1-\lambda_1\vert^\alpha + \frac{1}{4}\vert \lambda_1 \vert^\alpha \, \to \, \min\limits_{\lambda_1}.
\end{eqnarray*}
The last equation implies that the minimum is attained for some $\lambda_1 \in [0,1]$ and we can therefore replace the absolute values by brackets. By setting the derivative of the last equation with respect to $\lambda_1$ equal to zero, we obtain the LSL predictor
$$\widehat{X(t_0)} = \frac{1}{1+(1/3)^{1/(\alpha-1)}} X(t_1).$$
The COL predictor is
$$\widehat{X(t_0)} = \frac{[X(t_0),X(t_1)]_\alpha}{[X(t_1),X(t_1)]_\alpha} X(t_1) = \frac{3}{4} X(t_1)$$
and thus differs from the LSL predictor except for $\alpha=2$.
\end{example}

\subsubsection{Properties of the COL predictor for moving averages}\label{sec:COL_properties_MA}

In this subsection, we restrict our setting to moving average random fields of the form
\begin{equation}
X(t) = \int_{\R^d} f(t-x) M(dx), \quad t \in \R^d, \label{eq:MA}
\end{equation}
where $M$ is an $\alpha$-stable random measure with Lebesgue control measure. This restriction allows for imposing conditions on the kernel functions that guarantee the existence, uniqueness and continuity of the COL predictor.

We now state some properties of such moving average random fields. Let $h \in \R^d$. Since $X$ is stationary, it holds $[X(h),X(0)]_\alpha = [X(t+h),X(t)]_\alpha$ for each $t \in \R^d$. Therefore, the covariation function of $X$ is a function of just one variable and we use the notation $\kappa(h) = [X(h),X(0)]_\alpha$.

\begin{lem}\label{lemma:covariation_function}
If the kernel function $f:\R^d \to \R_+$ of the moving average (\ref{eq:MA}) is positive semi-definite, then the covariation function $\kappa$ is positive semi-definite. If $f:\R^d \to \R_+$ is positive definite and positive on a set with positive Lebesgue measure, then $\kappa$ is positive definite.
\end{lem}

It is clear that the COL predictor is unique if the covariation function is positive definite since the matrix of the system of linear equations
\begin{equation}
 \begin{pmatrix}
 \kappa(0) & \cdots & \kappa(t_n-t_1) \\
 \vdots & \ddots & \vdots \\
 \kappa(t_n-t_1) & \cdots & \kappa(0)\end{pmatrix} \begin{pmatrix} 
		\lambda_1 \\
		\vdots \\
		\lambda_n
	      \end{pmatrix} = \begin{pmatrix}
				\kappa(t_0-t_1) \\
				\vdots \\
				\kappa(t_0-t_n)
\end{pmatrix}\label{eq:soc4}
\end{equation}
is positive definite. An example of an $\alpha$-stable moving average with a positive definite kernel function is the $S\alpha S$ Ornstein-Uhlenbeck process with
$$X(t) = \int_\R e^{-\lambda(t-x)}\ind(t-x\geq 0) M(dx), \quad t \in \R,$$
for some $\lambda > 0$. If $t_1<t_2<\ldots<t_n<t_0$, then the regression of $X(t_0)$ on $(X(t_1),\ldots,X(t_n))^\mathsf{T}$ is linear (cf.~\cite[p.~138]{ST94}), i.~e. $\widehat{X(t_0)}=e^{-\lambda(t_{0}-t_n)} X(t_n)$.

\begin{thm}
If the covariation function is positive definite and continuous, then the COL predictor is continuous.
\end{thm}
The proof is analogous to the one of Theorem~\ref{th:COL_continuity}. For example, all continuous kernel functions with compact support yield a continuous covariation function by the dominated convergence theorem.

\subsection{Maximization of Covariation Linear (MCL) predictor}\label{subsec:MCL}

The idea of this method is to maximize the covariation
$$\left[\widehat{X(t_0)},X(t_0)\right]_\alpha = \sum_{i=1}^n \lambda_i \left[X(t_i),X(t_0)\right]_\alpha$$
between the linear predictor $\widehat{X(t_0)}$ and the theoretical value $X(t_0)$. In general, if we do not impose any further restrictions and if $\left[X(t_i),X(t_0)\right]_\alpha \neq 0$ for some $i \in \{1,\ldots,n\}$, then the maximization problem
\begin{equation}
\sum_{i=1}^n \lambda_i \left[X(t_i),X(t_0)\right]_\alpha \, \to \, \underset{\lambda_1,\ldots,\lambda_n}{\max} \label{eq:maxCov}
\end{equation}
obviously does not have a solution. A reasonable side condition is
$$\sigma_{\widehat{X(t_0)}} = \sigma_{X(t_0)}$$
which guarantees $\widehat{X(t_0)} \stackrel{d}{=} X(t_0)$ for $S\alpha S$ random fields $X$. The optimization problem becomes
\begin{equation}
 \begin{cases}
  \left[\widehat{X(t_0)},X(t_0)\right]_\alpha \, \to \, \underset{\lambda_1,\ldots,\lambda_n}{\max}, \\
  \sigma_{\widehat{X(t_0)}} = \sigma_{X(t_0)}.
 \end{cases}\label{MCL}
\end{equation}
The predictor $\widehat{X(t_0)}=\sum_{i=1}^n \lambda_i X(t_i)$ whose weights $\lambda_1,\ldots,\lambda_n$ are a solution of the maximization problem (\ref{MCL}) is called \textit{Maximization of Covariation Linear (MCL) predictor}.

\subsubsection{Properties of the MCL predictor}

The following theorem provides sufficient conditions for the existence, uniqueness, exactness and continuity of the MCL predictor.
\begin{thm}\label{th:MCL}
Assume that the random vector $(X(t_1),\ldots,X(t_n))^\mathsf{T}$ is full-dimensional and $t_0 \in \R^d$. Then the following holds:
\begin{enumerate}
\item The MCL predictor exists.
\item If additionally $\kappa(t_i,t_0) \neq 0$ for some $i\in\{1,\ldots,n\}$, then the MCL predictor at $t_0$ is unique.
\item If the MCL predictor is unique, then it is exact.
\item If $\kappa(t_i,t_0) \neq 0$ for some $i \in \{1,\ldots,n\}$ and $\kappa$ is continuous, then the MCL predictor is continuous in $t_0$.
\end{enumerate}
\end{thm}

\begin{remark}
We have already seen in Section \ref{sec:COL_properties_MA} that for moving averages, a sufficient condition for the continuity of the covariation function is that the kernel function is continuous and has a compact support.
\end{remark}

\section{Numerical results}

In this section, we apply the extrapolation methods to a sub-Gaussian random field and an $S \alpha S$ L\'evy motion on the square $[0,1]^2$. For this, we generate 1000 realizations of both random fields on an equidistant grid of $100\times 100$ points and extrapolate each realization by using the values $x(t_1),\ldots,x(t_9)$ at the locations $t_1=(0.2,0.2)^\mathsf{T}$, $t_2=(0.2,0.5)^\mathsf{T}$, $t_3=(0.2,0.8)^\mathsf{T}$, $t_4=(0.5,0.2)^\mathsf{T}$, $t_5=(0.5,0.5)^\mathsf{T}$, $t_6=(0.5,0.8)^\mathsf{T}$, $t_7=(0.8,0.2)^\mathsf{T}$, $t_8=(0.8,0.5)^\mathsf{T}$ and $t_9=(0.8,0.8)^\mathsf{T}$.

The sub-Gaussian random field $X=\{A^{1/2}G(t),t \in [0,1]^2\}$ has the stability index $\alpha=1.5$, where $G$ is a stationary Gaussian random field with Gaussian covariance function $C(h) = 7\exp\{-(h/0.1)^2\}$, $h \geq 0$. Recall that in this case, the LSL, COL and ML predictors coincide by Theorem~\ref{th:ML_LSL_COL}. Due to the specific structure of sub-Gaussian random fields, we also provide a conditional simulation (CS) of $X$, i.~e. a simulation of $X$ that honours the conditions $X(t_i)=x(t_i)$ for $i=1,\ldots,9$, see \cite{Lan02}. The conditional simulation can be seen as another extrapolation method and carried out straightforward by conditionally simulating the Gaussian part $G$ of $X$ and scaling it appropriately with a realization of $A^{1/2}$. A conditional simulation algorithm for Gaussian random fields is given in \cite{Lan02}. For a simulation algorithm for $A$, see \cite{CMS76}. Figure \ref{fig:subGaussian} shows a realization of $X$ and the extrapolations based on the LSL (COL, ML), MCL and CS methods.

\begin{figure}[ht!]
\includegraphics[width=7cm]{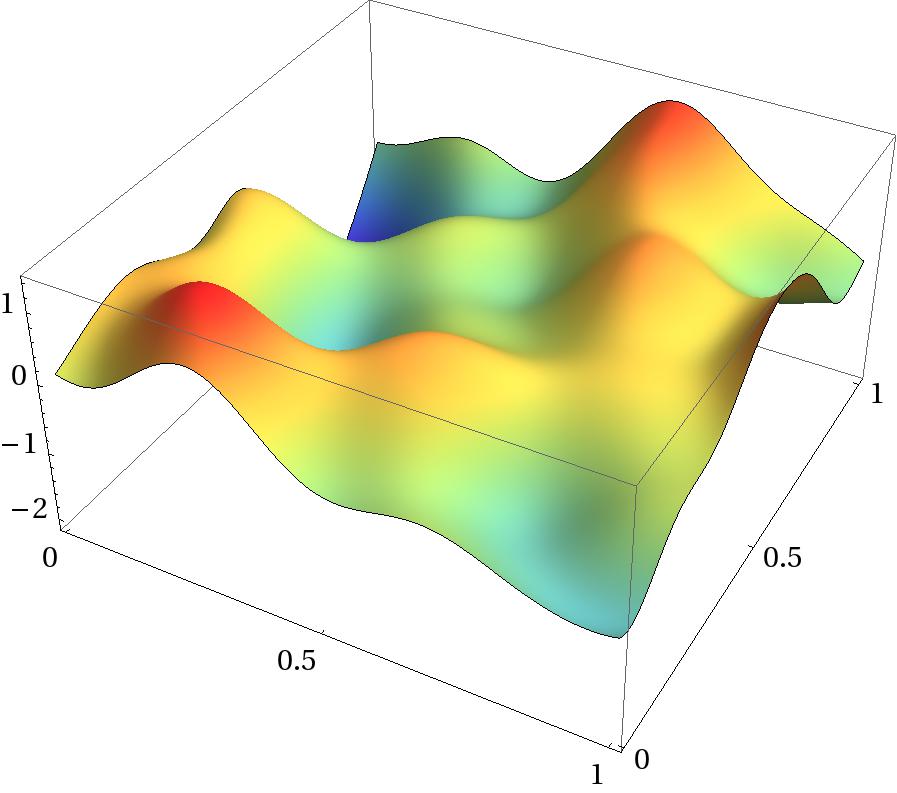}\hspace*{0.5cm} \includegraphics[width=7cm]{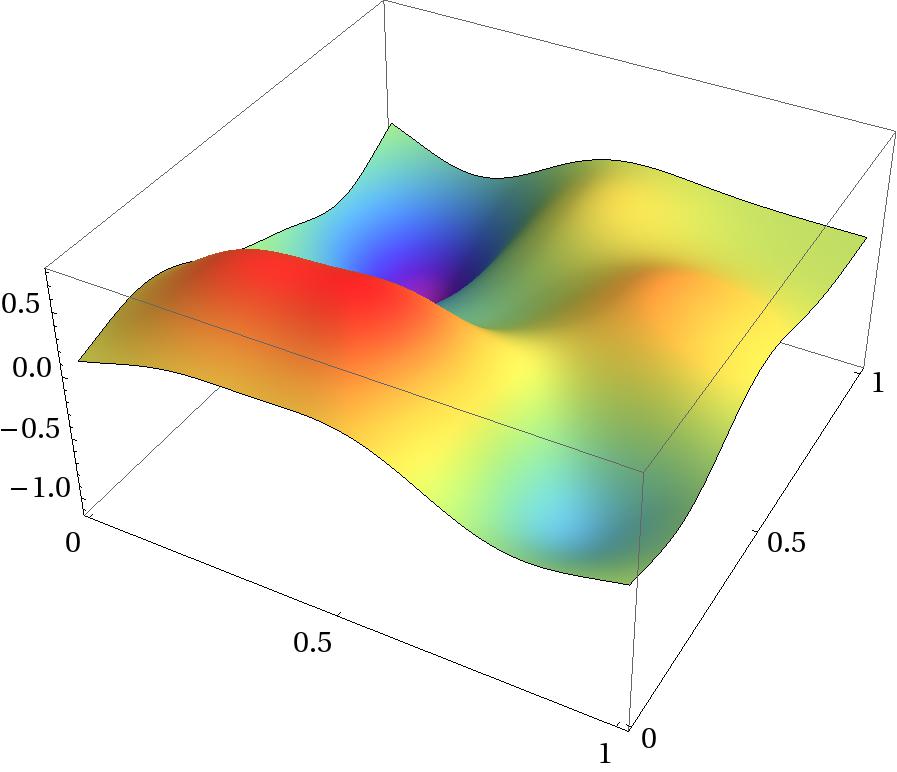} \\
\includegraphics[width=7cm]{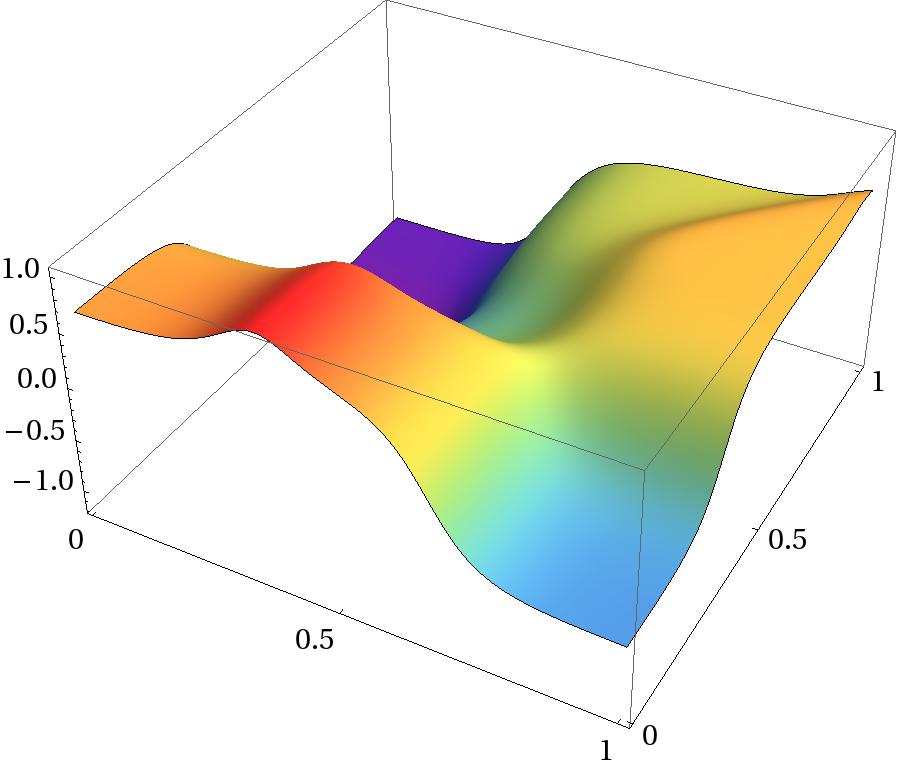}\hspace*{0.5cm} \includegraphics[width=7cm]{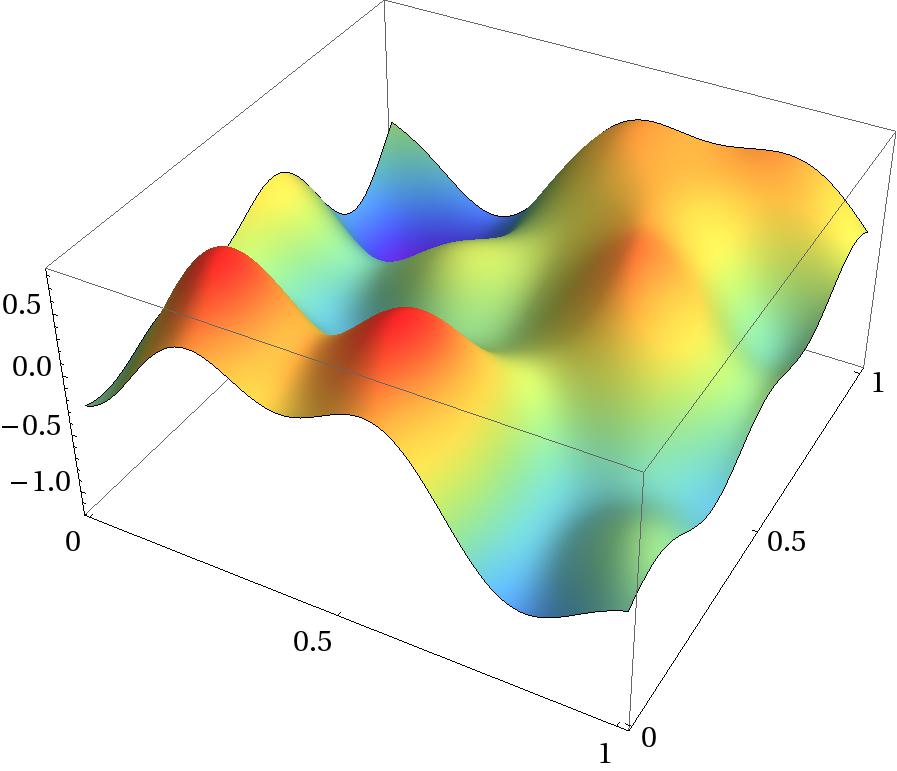} \\
\caption{Realization of the sub-Gaussian random field (top left) and the extrapolations based on the LSL (COL, ML) method (top right), the MCL method (bottom left) and the CS method (bottom right).}
\label{fig:subGaussian}
\end{figure}

\begin{table}[htbp]
	\newcolumntype{d}[1]{D{.}{.}{#1}}
	\setlength{\tabcolsep}{10pt}
	\footnotesize
	\vspace{1.5mm}
	\caption{Summary statistics for the deviations $X(g)-\widehat{X(g)}$ (sub-Gaussian random field)}
	    \begin{center}
        \begin{tabular}{lcccccc}
	    \toprule
		Method & 5\%-Quantile & 1st Quartile & Median & Mean & 3rd Quartile & 95\%-Quantile \\
	    \midrule
		LSL (COL, ML) & -1.5451 & -0.4446 & 0.0018 & 0.0070 & 0.4503 & 1.5363 \\
		MCL & -1.8204 & -0.4899 & 0.0046 & -0.0188 & 0.5016 & 1.7580 \\
		CS & -2.7523 & -0.5837 & 0.0058 & 0.0057 & 0.5985 & 2.7262 \\
        \bottomrule
        \end{tabular}\\
        \end{center}
	\label{tab:subGaussian}
\end{table}

In order to assess the performance of the predictors, we calculated the deviations $X(g)-\widehat{X(g)}$ of the extrapolated values from the simulated ones at each grid point $g$ of the 1000 realizations of the sub-Gaussian random field. The summary statistics of the deviations are shown in Table~\ref{tab:subGaussian}. As one can see, there are only marginal differences between the quantiles and the mean of the deviations for the three methods, the LSL (COL, ML) and MCL method performing slightly better than the CS method.

Consider now the two-dimensional $S\alpha S$ L\'evy motion $Y$ with stablility index $\alpha=1.5$ on the square $[0,1]^2$ defined by $Y(t) = \int_{[0,1]^2} \ind\{x_1\leq t_1,x_2 \leq t_2\} M(dx)$ for $t \in [0,1]^2$, where $M$ is an $S \alpha S$ random measure with Lebesgue control measure. Figure \ref{fig:LevyMotion} shows a realization of $X$ and the extrapolations based on the LSL, COL and MCL methods.

\begin{figure}[ht!]
\includegraphics[width=7cm]{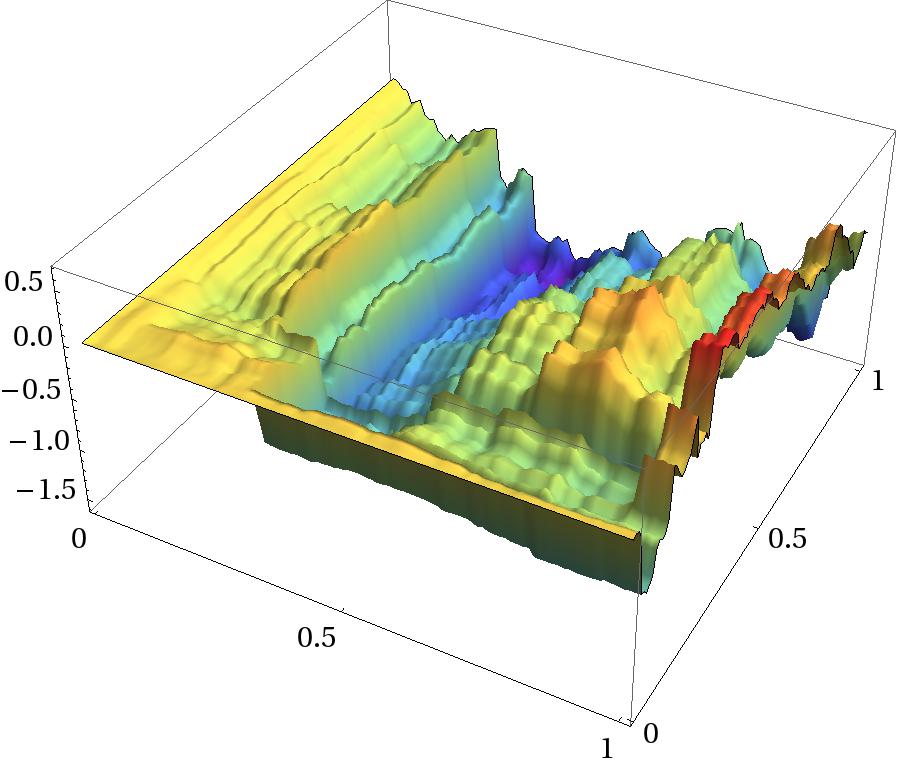}\hspace*{0.5cm} \includegraphics[width=7cm]{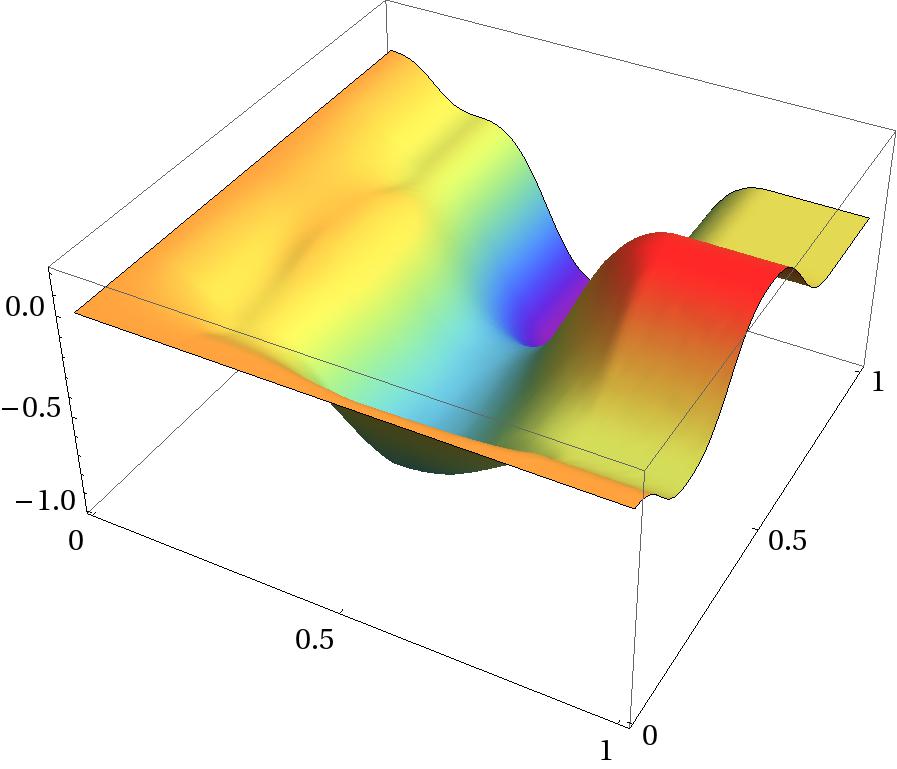} \\
\includegraphics[width=7cm]{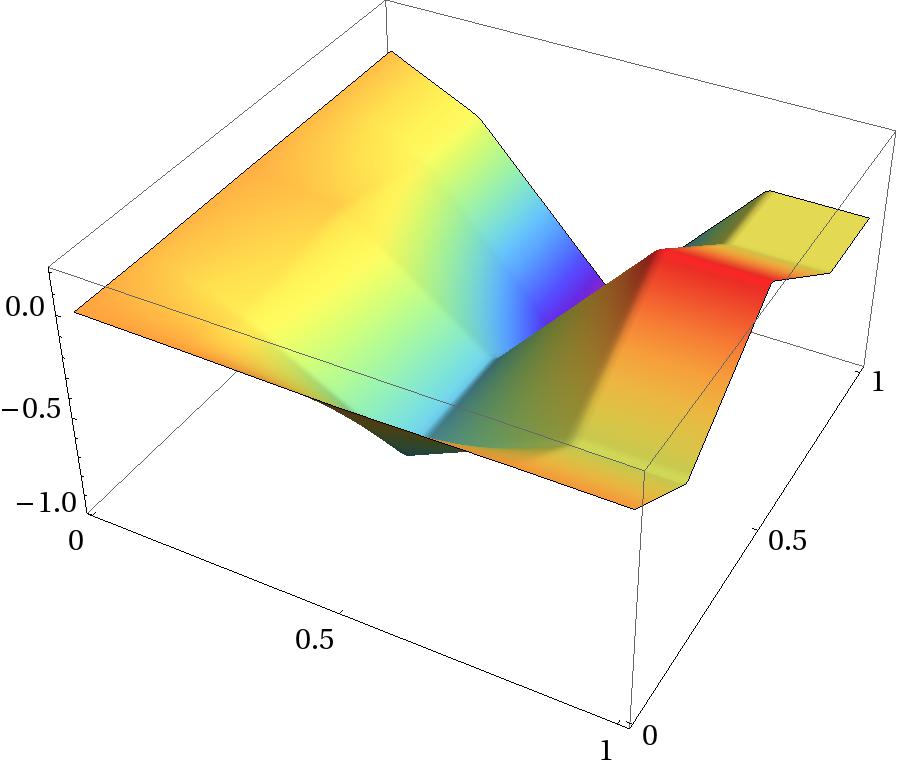}\hspace*{0.5cm} \includegraphics[width=7cm]{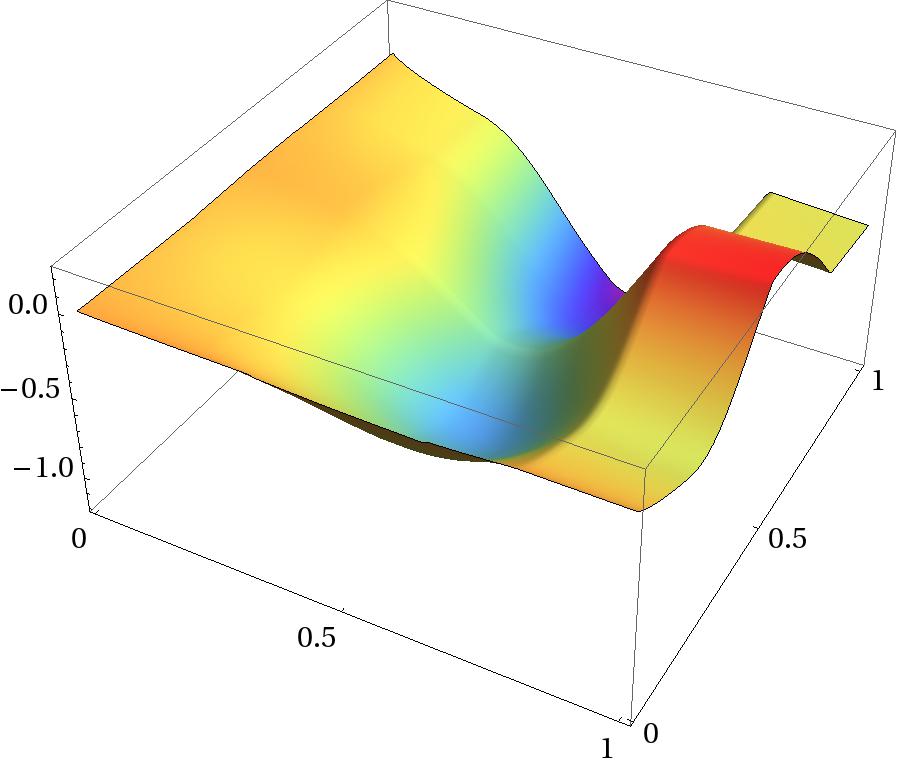} \\
\caption{Realization of the L\'evy stable motion (top left) and the extrapolations based on the LSL method (top right), the COL method (bottom left) and the MCL method (bottom right).}
\label{fig:LevyMotion}
\end{figure}

\begin{table}[ht!]
	\newcolumntype{d}[1]{D{.}{.}{#1}}
	\setlength{\tabcolsep}{10pt}
	\footnotesize
	\vspace{1.5mm}
	\caption{Summary statistics for the deviations $X(g)-\widehat{X(g)}$ ($S\alpha S$ L\'evy motion)}
	    \begin{center}
        \begin{tabular}{lcccccc}
	    \toprule
		Method & 5\%-Quantile & 1st Quartile & Median & Mean & 3rd Quartile & 95\%-Quantile \\
	    \midrule	    
		LSL & -0.5170 & -0.1246 & 0.0000 & -0.0071 & 0.1226 & 0.5045 \\
		COL & -0.5263 & -0.1289 & 0.0002 & -0.0061 & 0.1266 & 0.5137 \\
		MCL & -0.6093 & -0.1455 & -0.0007 & 0.0283 & 0.1407 & 0.5895 \\
        \bottomrule
        \end{tabular}\\
        \end{center}
	\label{tab:LevyMotion}
\end{table}

Again, we investigate the performance of the extrapolation methods by calculating the summary statistics of the deviations of the extrapolated values from the simulated ones. Table \ref{tab:LevyMotion} shows the corresponding results. As for the sub-Gaussian random field $X$, no method clearly outperforms the others.

\section{Discussion and open problems}

For sub-Gaussian random fields, the LSL and COL methods can be efficiently implemented and easily applied to larger data sets of at least $1000$ sample points since efficient algorithms to solve linear systems of equations are available. For other kinds of $\alpha$-stable random fields, the COL predictor can still be implemented efficiently for large data sets since only a system of linear equations has to be solved. In the LSL and MCL methods, a function whose dimension equals the number of sample points needs to be numerically minimized and maximized, respectively, impeding the application of these two methods to large data sets.

We notice that we have proven the properties of the LSL, COL and MCL predictors for $\alpha$-stable random fields with integral representation (\ref{eq:integral_representation}) and $\alpha>1$. All of these methods do not involve the skewness of the random field $X$. If $X$ is $S\alpha S$, then obviously so is $\hat{X}$ as it is a linear combination of $S\alpha S$ random variables. If the skewness function $\beta$ of $X$ is not zero, then easy calculations show that the skewness parameter of $\hat{X}(t) = \sum_{i=1}^n \lambda_i X(t_i)$ is given by
$$\frac{\int_E \left(\sum_{i=1}^n \lambda_i f_{t_i}(x)\right)^{\langle\alpha\rangle} \beta(x) m(dx)}{\int_E \left\vert\sum_{i=1}^n \lambda_i f_{t_i}(x)\right\vert^\alpha m(dx)},$$
cf.~\cite[Property~3.2.2]{ST94}.

The ML and CS predictors can be immediately extended to the case $\alpha \leq 1$, but they are only available for Gaussian and sub-Gaussian random fields. Also, the minimization problem (\ref{eq:LSL}) in the LSL method can be considered for $\alpha \leq 1$, but except for the case $\alpha=1$, the objective function will change from convexity to concavity in $\lambda=(\lambda_1,\ldots,\lambda_n)^\mathsf{T}$ and it is not clear whether the properties of the LSL predictor still hold. The development of extrapolation methods for $\alpha \leq 1$ is thus an interesting problem for future research.

Another interesting open problem is a characterization of the covariation function which has been touched upon only slightly in Lemma~\ref{lemma:covariation_function}.

\section*{Acknowledgement}

The authors are grateful to Ilya Molchanov for inspiring discussions on multivariate stable laws.

\appendix

\section{Proofs}\label{Proofs}

\begin{proof} [Proof of Lemma~\ref{lemma:lin_dep}]
Let $supp(\Gamma)$ denote the support of $\Gamma$ and assume that $supp(\Gamma) \subset \mathbb{S}^{d-1} \cap G_{c_1,\ldots,c_d}$ for some $(d-1)$-dimensional subspace
$G_{c_1,\ldots,c_d} := \{(x_1,\ldots,x_{d})^\mathsf{T} \in \R^{d}: c_1 x_1 + \ldots + c_{d} x_{d} = 0\}$
with $c_1,\ldots,c_{d} \in \R$. Let us consider $\sum_{i=1}^d c_i X_i$. Using \cite[Example~2.3.4]{ST94}, we obtain
\begin{eqnarray*}
\sigma_{\sum_{i=1}^d c_i X_i}^\alpha = \int_{\Sd} \left\vert \sum_{i=1}^d c_i s_i \right\vert^\alpha \Gamma(ds)= \int_{\Sd \cap G_{c_1,\ldots,c_d}} \left\vert \sum_{i=1}^d c_i s_i \right\vert^\alpha \Gamma(ds) = 0
\end{eqnarray*}
Since $\boldsymbol{\mu}=0$, we have $\sum_{i=1}^d c_i X_i = 0$ almost surely.

Vice versa, if $\sum_{i=1}^d c_i X_i = 0$ almost surely, then
$
\sigma_{\sum_{i=1}^d c_i X_i}^\alpha = \int_{\Sd} \left\vert \sum_{i=1}^d c_i s_i \right\vert^\alpha \Gamma(ds) = 0
$. Thus $\Gamma$ is concentrated on the great sub-sphere $G_{c_1,\ldots,c_d}:=\{(s_1,\ldots,s_d)^\mathsf{T} \in \Sd: \sum_{i=1}^d c_i s_i=0\}$ of $\Sd$.
\end{proof}

\begin{proof}[Proof of Lemma~\ref{lemma:association_kernels}]
Let $E_+:= \{x \in E: f_1(x)^2+f_2(x)^2 > 0\}$ and $g_j:E \to \R$ be defined by
$$g_j(x):= \frac{f_j(x)}{(f_1(x)^2+f_2(x)^2)^{1/2}},\quad x \in E, \quad j=1,2.$$
Then we have for any Borel set $A \subset \mathbb{S}^1$
$$\Gamma(A) = \int_{g^{-1}(A)} \frac{1+\beta(x)}{2}\left(f_1(x)^2+f_2(x)^2\right)^{\alpha/2}m(dx) + \int_{g^{-1}(-A)} \frac{1-\beta(x)}{2}\left(f_1(x)^2+f_2(x)^2\right)^{\alpha/2}m(dx),$$
where $g^{-1}(A):=\{x \in E_+:(g_1(x),g_2(x))^\mathsf{T} \in A\}$, see~\cite[pp.~115]{ST94}. Let $S^-=\{(s_1,s_2)^\mathsf{T}\in \mathbb{S}^1: s_1s_2<0\}$. It holds $S^-=-S^-$. The random vector $(X,Y)^\mathsf{T}$ is associated if and only if
$$\Gamma(S^-)=\int_{g^{-1}(S^-)} \left(f_1(x)^2+f_2(x)^2\right)^{\alpha/2}m(dx) = 0,$$
see~\cite[Theorem~4.6.1]{ST94}. This is equivalent to $m(g^{-1}(S^-))=0$, that is $f_1f_2\geq 0$ $m$-almost everywhere. The proof for negative association is analogous by using \cite[Theorem~4.6.3]{ST94}.
\end{proof}

\begin{proof}[Proof of Corollary~\ref{corollary:decomposition}]
By~\cite[Theorem~3.5.6]{ST94}, $(X,Y)^\mathsf{T}$ has an integral representation
\begin{equation*}
(X,Y)^\mathsf{T} \stackrel{d}{=} \left(\int_E f_1(x)M(dx), \int_E f_2(x) M(dx)\right)^\mathsf{T},
\end{equation*}
where $M$ is an $\alpha$-stable random measure with control measure $m$ and skewness intensity function $\beta$. We define $E^*:= \{x \in E: f_1(x)f_2(x) \geq 0\}$, $X_1:=\int_{E} f_1(x)\ind_{E^*}(x)M(dx)$, $Y_1:=\int_{E} f_2(x)\ind_{E^*}(x)M(dx)$, $X_2:=X-X_1$ and $Y_2:=Y-Y_1$. The dependence properties of $X_1$, $X_2$, $Y_1$ and $Y_2$ follow immediately from Lemma~\ref{lemma:association_kernels} and the fact that an $\alpha$-stable random vector $(U,V)^\mathsf{T} = \left(\int_E g_1(x)M(dx), \int_E g_2(x) M(dx)\right)^\mathsf{T}$ has independent components if and only if $g_1g_2 = 0$ $m$-almost everywhere, see~\cite[Theorem~3.5.3]{ST94}..
\end{proof}

\begin{proof}[Proof of Lemma~\ref{lemma:mixed_moments}]
By~\cite[Property~1.2.17]{ST94}, we have $\E |\lambda X+Y|^p= \E |\xi|^p \sigma_{\lambda X+Y}^p$ for $\lambda\in \R$,
where $\xi\sim S_{\alpha}(1, \beta_{\lambda X+Y}, 0)$ and $\beta_{\lambda X+Y}={\int_{\mathbb{S}^1} (\lambda s_1+s_2)^{<\alpha>}\Gamma(ds)}\left/{\int_{\mathbb{S}^1} |\lambda s_1+s_2|^{\alpha}\Gamma(ds)}\right.$. Taking derivatives on both sides, we get
\begin{eqnarray}
\left.\frac{d}{d\lambda} \E |\lambda X+Y|^p\right\vert_{\lambda=0}&=&p\E \left(X Y^{<p-1>}\right), \nonumber\\
\left.\frac{d}{d \lambda}\E |\xi|^p \sigma_{\lambda X+Y}^p \right\vert_{\lambda=0}&=&p \sigma_Y^{p-\alpha}[X,Y]_{\alpha} \left.\E |\xi|^p\right\vert_{\lambda=0}+  \sigma_Y^{p} \left.\frac{d}{d\lambda}\E |\xi|^p\right\vert_{\lambda=0}. \label{eq:moment_derivative}
\end{eqnarray}
A closed formula for $\E |\xi|^p$ is given in \cite[p.~18]{ST94}. After some cumbersome calculations, we obtain 
\begin{eqnarray*}
\left.\E |\xi|^p\right\vert_{\lambda=0} &=& \E |Y|^p/\sigma^p_Y, \\
\left.\frac{d}{d\lambda}\E |\xi|^p\right\vert_{\lambda=0}&=& p \frac{\E |Y|^p}{\sigma^{p+\alpha}_Y} \cdot c\cdot \left((X,Y)_{\alpha}- \beta_Y[X,Y]_{\alpha}\right).
\end{eqnarray*}
Plugging these formulas in (\ref{eq:moment_derivative}) and equating both derivatives yields the assertion.
\end{proof}

\subsection{Proofs for the LSL predictor}
\begin{proof}[Proof of Lemma \ref{lemma:first_order_conditions}]
 Let $\lambda_0 := -1$. Then we get for $j=1,\ldots,n$
 \begin{eqnarray*}
  \sigma_{\widehat{X(t_0)}-X(t_0)}^\alpha = \sigma_{\sum_{i=0}^n \lambda_i X(t_j)}^\alpha 
  = \int_{\mathbb{S}^n} \left\vert \sum_{i=0}^n \lambda_i s_i \right\vert^\alpha \Gamma(ds),
 \end{eqnarray*}
 where $\Gamma$ is the spectral measure of the random vector $(X(t_0),\ldots,X(t_n))^\mathsf{T}$, cf.~\cite[formula~(2.3.3)]{ST94}. It is easy to see that for $j=1,\ldots,n$, the partial derivative
$
 \frac{d}{d\lambda_j}\left\vert\sum_{i=0}^n \lambda_i s_i\right\vert^\alpha
$
(as a function of $\lambda_j$) is uniformly bounded in $(s_0,\ldots,s_n)^\mathsf{T} \in \mathbb{S}^n$ if $\alpha > 1$. Then since $\Gamma$ is a finite measure, it follows from the dominated convergence theorem that
 \begin{eqnarray*}
 &&\frac{\partial}{\partial \lambda_j} \sigma_{\widehat{X(t_0)}-X(t_0)}^\alpha   = \int_{\mathbb{S}^n} \frac{\partial}{\partial \lambda_j} \left\vert \sum_{i=0}^n \lambda_i s_i \right\vert^\alpha \Gamma(ds) \\
  &=&\alpha \int_{\mathbb{S}^n} s_j \left( \sum_{i=0}^n \lambda_i s_i \right)^{< \alpha-1>} \Gamma(ds) 
  = \alpha \left[X(t_j),\sum_{i=1}^n\lambda_i X(t_i) - X(t_0) \right]_\alpha,
 \end{eqnarray*}
 see~\cite[Lemma~2.7.5]{ST94}. 
\end{proof}

\begin{proof}[Proof of Theorem \ref{th:LSLP1}]
Assertion 1 follows from \cite[Theorem~1.1, p.~59]{DeVL93}, the discussion thereafter and Lemma~\ref{lemma:great_subsphere}. The exactness of the predictor is obvious. We now prove the second assertion. Since $X$ is stochastically continuous, we have by (\ref{eq:integral_scale}) and (\ref{eq:stochastic_convergence})
\begin{equation*}
\Vert f_s - f_t \Vert_{L^\alpha(E,m)} \to 0, \quad s \to t,
\end{equation*}
where $\Vert \cdot \Vert_{L^\alpha(E,m)}$ denotes the usual norm of $L^\alpha(E,m)$. Let $\hat{f}_t = \sum_{i=1}^n \lambda_i(t) f_{t_i}$ for each $t \in \R^d$. Then it holds
\begin{equation}
\left\Vert \hat{f}_s - \hat{f}_t\right\Vert_{L^\alpha(E,m)} = \left\Vert \sum_{i=1}^n (\lambda_i(s)-\lambda_i(t)) f_{t_i}\right\Vert_{L^\alpha(E,m)} \to 0, \quad s \to t, \label{eq:continuity_bestApproximation}
\end{equation}
see \cite[Theorem~1.2, p.~60]{DeVL93}.

Assume that $\lambda_i(s)\nrightarrow \lambda_i(t)$ as $s \to t$ for some $i \in \{1,\ldots,n\}$. Then $\sqrt{(\lambda_1(s)-\lambda_1(t))^2+\cdots+(\lambda_n(s)-\lambda_n(t))^2} \nrightarrow 0$ as $s \to t$. Since
\begin{equation*}
\left\vert \frac{\lambda_i(s)-\lambda_i(t)}{\sqrt{(\lambda_1(s)-\lambda_1(t))^2+\cdots+(\lambda_n(s)-\lambda_n(t))^2}} \right\vert \leq 1, \quad i=1,\ldots,n,
\end{equation*}
there exists a sequence $\{s_{k}\}_{k \in \N}$ with $s_k \to t$ as $k \to \infty$ such that
\begin{equation*}
\gamma_i^{(k)} := \frac{\lambda_i(s_{k})-\lambda_i(t)}{\sqrt{(\lambda_1(s_{k})-\lambda_1(t))^2+\cdots+(\lambda_n(s_{k})-\lambda_n(t))^2}} \to \gamma_i, \quad k \to \infty, \quad i=1,\ldots,n.
\end{equation*}
Notice that $(\gamma_1^{(k)})^2 = 1 - \sum_{i=2}^n (\gamma_i^{(k)})^2$ for each $k \in \N$ such that $\gamma_i \neq 0$ for some $i \in \{1,\ldots,n\}$. By (\ref{eq:continuity_bestApproximation}), we have $\Vert \sum_{i=1}^n \gamma_i f_{t_i} \Vert_{L^\alpha(E,m)}=0$ which implies $\sum_{i=1}^n \gamma_i f_{t_i} = 0$ $m$-almost everywhere. By Lemma~\ref{lemma:great_subsphere}, this is a contradiction to the the full-dimensionality of $(X(t_1),\ldots,X(t_n))^\mathsf{T}$.
\end{proof}

\subsection{Proofs for the COL predictor}

\begin{proof}[Proof of Theorem \ref{th:COL_continuity}]
Let $t_0 \in \R^d$. The matrix on the left hand side of equation (\ref{eq:soc3}) is a bijective continuous linear operator. By the open mapping theorem, the inverse operator is continuous. Since the covariance function is continuous, the weights $\lambda_1,\ldots,\lambda_n$ (considered as real-valued functions on $\R^d$) are continuous in $t_0$.
\end{proof}

\begin{proof}[Proof of Theorem \ref{th:ML_LSL_COL}]
Let $A$ be a lower triangular $(n+1)\times (n+1)$-matrix with the property
\begin{equation}
A \Omega_{t_1,\ldots,t_n,t_0} A^\mathsf{T} = 2I, \label{eq:A}
\end{equation}
where $I$ denotes the identity matrix and $\Omega_{t_1,\ldots,t_n,t_0}$ is the covariance matrix of the Gaussian part of the random vector $(X(t_1),\ldots,X(t_n),X(t_0))^\mathsf{T}$. Then the ML predictor for $X(t_0)$ is given by
\begin{equation}
\widehat{X(t_0)} = \begin{cases}
					\sum_{i=1}^n \frac{-a_{n+1,i}}{a_{n+1,n+1}} X(t_i), & t_0 \neq t_i \text{ for all }  i \in \{1,\ldots,n\}, \\
					X(t_i), & t_0 = t_i \text{ for some } i \in \{1,\ldots,n\}.
				\end{cases} \label{eq:ML_elliptical}
\end{equation}
where $a_{i,j}$ denotes the entry in the $i$-th row and $j$-th column of $A$, see \cite{Pai98}.
Let $C$ be the covariance function of the Gaussian part $G$ of the sub-Gaussian random field $X$. We have $\Omega_{t_1,\ldots,t_n,t_0} = 2(A^\mathsf{T} A)^{-1}$ and $A=B^{-1}$, where $B$ is the lower triangular matrix of the Cholesky decomposition of $\Omega_{t_1,\ldots,t_n,t_0}$, i.~e. $\Omega_{t_1,\ldots,t_n,t_0}=2 BB^\mathsf{T}$. Notice that the last column of the matrix $-1/a_{n+1,n+1}\cdot A^\mathsf{T}$ consists of the coefficients $-a_{n+1,1}/a_{n+1,n+1},\ldots,-a_{n+1,n}/a_{n+1,n+1}$ of the ML predictor and the number $-1$. Let $\lambda_i = -a_{n+1,i}/a_{n+1,n+1}$ for $i=1,\ldots,n$. We have
\begin{eqnarray*}
&&\frac{1}{2a_{n+1,n+1}}\Omega_{t_1,\ldots,t_n,t_0} A^\mathsf{T} = \frac{1}{a_{n+1,n+1}}(A^\mathsf{T}A)^{-1}A^\mathsf{T} = \frac{1}{a_{n+1,n+1}} A^{-1} (A^\mathsf{T})^{-1} A^\mathsf{T} \\
 &=& \frac{1}{a_{n+1,n+1}}A^{-1} = \frac{1}{a_{n+1,n+1}}B = \begin{pmatrix}
 x_{1,1} & 0 & \cdots & 0 \\
 \vdots & \ddots & \ddots & \vdots \\
 x_{n,1} &\cdots & x_{n,n}  & 0 \\
 x_{n+1,1} & \cdots & x_{n+1,n} & x_{n+1,n+1}
 \end{pmatrix}
\end{eqnarray*}
for some real numbers $x_{1,1},\ldots,x_{n+1,n+1}$. By only considering the last column of the left and right hand side of the last equation yields
\begin{equation*}
 \begin{pmatrix}
 C(0) & \cdots & C(t_n-t_1) \\
 \vdots & \ddots & \vdots \\
 C(t_n-t_1) & \cdots & C(0)\end{pmatrix} \begin{pmatrix} 
		\lambda_1 \\
		\vdots \\
		\lambda_n
	      \end{pmatrix} - \begin{pmatrix}
				C(t_0-t_1) \\
				\vdots \\
				C(t_0-t_n)
\end{pmatrix} = \begin{pmatrix}
				0 \\
				\vdots \\
				0
\end{pmatrix}
\end{equation*}
which is the system of linear equations that has to be solved for the COL predictor. Since the COL predictor is unique, it must be equal to the ML predictor.

We now show that the COL predictor is equal to the LSL predictor. Let $X = \{A^{1/2} G(t), t \in \R^d\}$ be a sub-Gaussian random field. We have
$
\widehat{X(t_0)} - X(t_0) = A^{1/2} \sum_{i=0}^n \lambda_i G(t_i),
$
with $\lambda_0 = -1$ and
$$\Var\left(\sum_{i=0}^n \lambda_i G(t_i)\right) = \sum_{i,j=0}^n \lambda_i\lambda_j C(t_i-t_j).$$
With \cite[Proposition~1.3.1]{ST94}, we get
$$\sigma_{\widehat{X(t_0)}-X(t_0)}^\alpha = \left( \frac{1}{2}\sum_{i,j=0}^n \lambda_i\lambda_j C(t_i-t_j)\right)^{\alpha/2} \, \to \, \min\limits_{\lambda_1,\ldots,\lambda_n}$$
which is equivalent to
$$\sum_{i,j=0}^n \lambda_i\lambda_j C(t_i-t_j) \, \to \, \min\limits_{\lambda_1,\ldots,\lambda_n}.$$
Taking partial derivatives and setting them equal to zero shows that the LSL predictor is equal to the COL predictor. The calculations for Gaussian random fields are analogous.
\end{proof}

\begin{proof}[Proof of Lemma \ref{lemma:covariation_function}]
Let $n \in \N$, $z_1,\ldots,z_n \in \R$ and $t_1,\ldots,t_n \in \R^d$. It holds
\begin{eqnarray*}
&&\sum_{i,j=1}^n \kappa(t_i-t_j) z_i z_j = \left[\sum_{i,j=1}^n X(t_i-t_j)z_i z_j,X(0)\right]_\alpha \\
&=& \int_{\R^d} \sum_{i,j=1}^n f(t_i-t_j-x) z_i z_j f^{\langle \alpha -1\rangle}(-x) dx \geq 0
\end{eqnarray*}
if $f$ is positive semi-definite and $f(x) \geq 0$ for all $x \in \R^d$. We consider $(z_1,\ldots,z_n)^\mathsf{T} \neq (0,\ldots,0)^\mathsf{T}$ for the positive definiteness of $\kappa$. Let $A$ be the set where $f$ is positive with positive Lebesgue measure. Then $-x \in -A$ if and only if $x \in A$ and
\begin{eqnarray*}
\sum_{i,j=1}^n \kappa(t_i-t_j) z_i z_j \geq \int_{-A} \sum_{i,j=1}^n  f(t_i-t_j-x) z_i z_j f^{\langle \alpha -1\rangle}(-x) dx > 0.
\end{eqnarray*}
\end{proof}

\subsection{Proofs for the MCL predictor}

\begin{proof}[Proof of Theorem \ref{th:MCL}]
\begin{enumerate}
\item Let $C:= \{ (\lambda_1,\ldots,\lambda_n)^\mathsf{T} \in \R^n: \sigma_{\sum_{i=1}^n \lambda_i X(t_i)} = \sigma_{X(t_0)}\}$. Let $\lambda_2=\ldots=\lambda_n=0$. Then
$\sigma_{\sum_{i=1}^n \lambda_i X(t_i)} = \sigma_{\lambda_1 X(t_1)} = \vert \lambda_1\vert \sigma_{X(t_1)}$
such that $\lambda_1 = \pm \sigma_{X(t_0)}/\sigma_{X(t_1)}$. Therefore, the set $C$ is non-empty. Since $L^\alpha(E,m)$ is a linear space, $\sigma_{\sum_{i=1}^n \lambda_i X(t_i)}$ is a continuous function in $\lambda=(\lambda_1,\ldots,\lambda_n)^\mathsf{T} \in \R^n$ by Lebesgue's theorem on dominated convergence. Thus the set $C$ is closed. By Lemma~\ref{lemma:great_subsphere}, the functions $f_{t_i}$, $i=1,\ldots,n$, are linearly independent. Due to this, $\sigma_{\sum_{i=1}^n \lambda_i X(t_i)}=\int_E\vert\sum_{i=1}^n\lambda_if_{t_i}(x)\vert^\alpha m(dx)$ tends to infinity as $\Vert \lambda \Vert_2 \to \infty$ so that $C$ is bounded and hence compact. Since the objective function is continuous, there exists a solution of the optimization problem.

\item Recall the optimization problem
\begin{equation}
 \begin{cases}
  \left[\widehat{X(t_0)},X(t_0)\right]_\alpha \, \to \, \underset{\lambda_1,\ldots,\lambda_n}{\max}, \\
  \sigma_{\widehat{X(t_0)}} = \sigma_{X(t_0)}.
 \end{cases} \label{eq:op}
\end{equation}
Define the function $\Psi: \R^n \to \R$ by
$$\Psi(\lambda) := \sigma_{\widehat{X(t_0)}} = \left\Vert\sum_{i=1}^n \lambda_i f_{t_i} \right\Vert_\alpha, \quad \lambda = (\lambda_1,\ldots,\lambda_n)^\mathsf{T} \in \R^n.$$
It is obvious that $\Psi$ is continuous. By Minkowski's inequality, $\Psi$ is convex and strictly convex on the set of admissible points given by $\Lambda:=\{\lambda=(\lambda_1,\ldots,\lambda_n)^\mathsf{T} \in \R^n: \Psi(\lambda) = \sigma_{X(t_0)}\}$, that is
$$\Psi(\beta \lambda^{(1)} + (1-\beta) \lambda^{(2)}) < \beta \Psi(\lambda^{(1)}) + (1-\beta) \Psi(\lambda^{(2)})$$
for all $\beta \in (0,1)$ and $\lambda^{(1)},\lambda^{(2)} \in \Lambda$ with $\lambda^{(1)} \neq \lambda^{(2)}$.

Indeed, let $\lambda^{(1)},\lambda^{(2)} \in \Lambda$ with $\lambda^{(1)} \neq \lambda^{(2)}$ and $\beta \in (0,1)$. We have
\begin{equation*}
\Psi(\beta \lambda^{(1)} + (1-\beta) \lambda^{(2)}) = \left\Vert \beta \sum_{i=1}^n \lambda_i^{(1)} f_{t_i} + (1-\beta) \sum_{i=1}^n \lambda_i^{(2)} f_{t_i} \right\Vert_\alpha.
\end{equation*}
Assume that
\begin{equation}
\beta \sum_{i=1}^n \lambda_i^{(1)} f_{t_i}(x) = \gamma (1-\beta) \sum_{i=1}^n \lambda_i^{(2)} f_{t_i}(x) \quad m-a.e. \label{eq:minkowski}
\end{equation}
for some $\gamma \geq 0$. By Lemma~\ref{lemma:great_subsphere}, this contradicts the full-dimensionality of the vector $(X(t_1),\ldots,X(t_n))^\mathsf{T}$. Analogously, by exchanging $\beta$ and $1-\beta$, we have
\begin{equation*}
(1-\beta) \sum_{i=1}^n \lambda_i^{(2)} f_{t_i}(x) \neq \gamma \beta \sum_{i=1}^n \lambda_i^{(1)} f_{t_i}(x) \quad m-a.e.
\end{equation*}
for all $\gamma \geq 0$. The Minkowski inequality yields
$\Vert g + h \Vert_\alpha \leq \Vert g \Vert_\alpha + \Vert h \Vert_\alpha$
for all $g,h \in L^\alpha(E,m)$ and equality holds if and only if $g=\gamma h$ or $h = \gamma g$ $m$-almost everywhere for some $\gamma \geq 0$. We conclude that
\begin{eqnarray*}
&&\Psi(\beta \lambda^{(1)} + (1-\beta) \lambda^{(2)}) = \left\Vert \beta \sum_{i=1}^n \lambda_i^{(1)} f_{t_i} + (1-\beta) \sum_{i=1}^n \lambda_i^{(2)} f_{t_i} \right\Vert_\alpha \\
&<& \beta \left\Vert \sum_{i=1}^n \lambda_i^{(1)} f_{t_i} \right\Vert_\alpha + (1-\beta) \left\Vert \sum_{i=1}^n \lambda_i^{(2)} f_{t_i} \right\Vert_\alpha = \beta \Psi(\lambda^{(1)}) + (1-\beta) \Psi(\lambda^{(2)})
\end{eqnarray*}
such that $\Psi$ is strictly convex on $\Lambda$.

Consider now the set $B := \{\lambda \in \R^n: \Psi(\lambda) \leq \sigma_{X(t_0)}\}$. $B$ is convex since $\Psi$ is convex. Moreover, it is strictly convex which can be shown as follows. Let $\lambda \in \partial B$, where $\partial B$ denotes the boundary of $B$. Then $\Psi(\lambda) = \sigma_{X(t_0)}$ because otherwise $\Psi(\lambda) < \sigma_{X(t_0)}$, and since $\Psi$ is continuous, there exists an $\varepsilon > 0$ such that $\Psi(x) < \sigma_{X(t_0)}$ for all $x \in B_\lambda(\varepsilon) := \{x \in \R^n: \Vert x -\lambda \Vert_2 \leq \varepsilon\}$. Therefore, $B_\lambda(\varepsilon) \subset B$ which is a contradiction to $\lambda \in \partial B$.

Let $\lambda^{(1)},\lambda^{(2)} \in \partial B$ and $\beta \in (0,1)$. We have already shown that
$$\Psi(\beta \lambda^{(1)} + (1-\beta) \lambda^{(2)}) < \beta \Psi(\lambda^{(1)}) + (1-\beta) \Psi(\lambda^{(2)}) = \sigma_{X(t_0)},$$
so $\beta \lambda^{(1)} + (1-\beta) \lambda^{(2)} \in \mathring{B}$, where $\mathring{B}$ denotes the interior of $B$. Therefore, $B$ is strictly convex.

The strict convexity of $B$ implies that for each $u \in \R^n$, the support set
\begin{equation}
T(B,u) = \{ y \in B: \langle y,u \rangle = h(B,u)\} \label{eq:supportSet}
\end{equation}
consists of a singleton $\{x\}$, cf. the proof of \cite[Theorem 4.5]{Mol09}, where
\begin{equation}
h(B,u) = \sup\{\langle x,u \rangle: x \in B\} \label{eq:supportFunction}
\end{equation}
is the support function of $B$. It is clear that $x \in \partial B$. Recall that the optimization problem (\ref{eq:op}) can be written as
\begin{equation*}
(P) \begin{cases}
	\langle \lambda,a \rangle \to \max \\
	\lambda \in \Lambda
	\end{cases}
\end{equation*}
where
\begin{eqnarray*}
a &:=& ([X(t_1),X(t_0)]_\alpha,\ldots,[X(t_n),X(t_0)]_\alpha)^\mathsf{T}, \\
\Lambda &:=& \{\lambda=(\lambda_1,\ldots,\lambda_n)^\mathsf{T} \in \R^n: \sigma_{\sum_{i=1}^n \lambda_i X(t_i)} = \sigma_{X(t_i)}\}.
\end{eqnarray*}
Consider the optimization problem
\begin{equation*}
(P^*) \begin{cases}
	\langle \lambda,a \rangle \to \max \\
	\lambda \in B
	\end{cases}
\end{equation*}
which has a unique solution $\lambda^*=T(B,a) \in \partial B$ by (\ref{eq:supportSet}) and (\ref{eq:supportFunction}) since $T(B,a)$ consists of only one element. We have already seen that $\partial B \subset \Lambda$ such that $\lambda^*$ is the unique solution of $(P)$, too.

\item The Lagrange function $L:\R^n \times \R \to \R$ of the optimization problem (\ref{eq:op}) is given by
\begin{equation*}
L(\lambda,\gamma) = \sum_{i=1}^n \lambda_i [X(t_i),X(t_0)]_\alpha + \gamma \left(\sigma_{\sum_{i=1}^n \lambda_i X(t_i)} - \sigma_{X(t_0)}\right).
\end{equation*}
By taking partial derivatives and setting them equal to zero, we get the system of equations
\begin{eqnarray*}
[X(t_j),X(t_0)]_\alpha + \gamma \left[X(t_j),\sum_{i=1}^n \lambda_i X(t_i)\right]_\alpha &=& 0, \quad j=1,\ldots,n, \\
\sigma_{\sum_{i=1}^n \lambda_i X(t_i)} &=& \sigma_{X(t_0)}.
\end{eqnarray*}
Suppose that $t_k = t_0$ for some $k \in \{1,\ldots,n\}$. Then $(\lambda^*,\gamma^*)$ with
\begin{eqnarray*}
\gamma^* = -1 \quad \text{and} \quad \lambda_i^* = \begin{cases}
			 1, & i = k, \\
			 0, & i \neq k,
			\end{cases} \quad i=1,\ldots,n,
\end{eqnarray*}
is a solution of the system of equations and therefore a critical point of the objective function. Since the MCL predictor is unique, it is given by $\widehat{X(t_0)} = \sum_{i=1}^n \lambda_i^* X(t_i)$.

\item Let $C$ be an arbitrary set and $g:C \to \R$ be a function. We set
\begin{align*}
&\inf g(C) := \inf\{g(x):x \in C\}, \\
&M(g,C) :=\{x \in C: g(x) = \inf g(C)\}.
\end{align*}
For a set $M \subset \R^n$ and a sequence $(M_l)_{l \in \N}$ of subsets of $\R^n$, we define
\begin{eqnarray}
M=\lim_{l \to \infty}M_l \, \Leftrightarrow \, \lim_{l \to \infty} d(x,M_l) = 0,\, x \in M \text{ and } \liminf_{l \to \infty} d(x,M_l) > 0,\, x \notin M, \label{eq:set_convergence}
\end{eqnarray}
where $d(x,M) := \inf_{m \in M} \Vert x-m\Vert_2$, cf.~\cite[Anhang~1, Bemerkung~1]{Kos91}. Notice that this definition is equivalent to the Hausdorff convergence, see~\cite[Anhang~1, Satz~4]{Kos91}. The following fact provides sufficient conditions for the stability of the solution of an optimization problem, see~\cite[Stabilit\"atssatz~5.4.2]{Kos91}.\vspace*{0.2cm}\\
\textit{
Let $(S_m)_{m \in \N}$ be a sequence of closed convex subsets of $\R^n$, $n \in \N$, with $\lim_{m \to \infty}S_m := S$. Let $(\Phi_m:\R^n \to \R)_{m \in \N}$ be a sequence of convex functions which converges pointwise to some function $\Phi:\R^n \to \R$. If $M(\Phi,S)$ consists of a singleton $\{x\}$, then $x_m \in M(\Phi_m,S_m)$, $m \in \N$, implies $\lim_{m \to \infty} x_m = x$.
}\vspace*{0.2cm}\\
Let $(s_m)_{m \in \N} \subset \R^d$ with $\lim_{m \to \infty} s_m = t_0$. We set
\begin{eqnarray*}
&&S_m := \left\{(\lambda_1,\ldots,\lambda_n)^\mathsf{T} \in \R^n:\int_E \left\vert\sum_{i=1}^n \lambda_i f_{t_i}(x)\right\vert^\alpha m(dx) \leq \int_E \left\vert f_{s_m}(x)\right\vert^\alpha m(dx) \right\}, \\
&&S := \left\{(\lambda_1,\ldots,\lambda_n)^\mathsf{T} \in \R^n:\int_E \left\vert\sum_{i=1}^n \lambda_i f_{t_i}(x)\right\vert^\alpha m(dx) \leq \int_E \left\vert f_{t_0}(x)\right\vert^\alpha m(dx) \right\}, \\
&&\Phi_m(\lambda) := -\sum_{i=1}^n \lambda_i [X(t_i),X(s_m)]_\alpha,\\
&&\Phi(\lambda) := -\sum_{i=1}^n \lambda_i [X(t_i),X(t_0)]_\alpha.\\
\end{eqnarray*}
We notice that we have shown in the second part of the proof that the constraint $\sigma_{\widehat{X(t_0)}} = \sigma_{X(t_0)}$ can be replaced by $\sigma_{\widehat{X(t_0)}} \leq \sigma_{X(t_0)}$. Thus, we can consider $S$ as the set of admissible points of the MCL optimization problem. Furthermore, the previous fact is formulated for minimization problems, but the MCL predictor requires maximization. Therefore, the functions $\Phi_m$ and $\Phi$ are defined with a negative sign to adjust for this fact. Since the covariation function is continuous, the sequence $(\Phi_m)_{m \in \N}$ converges pointwise to $\Phi$. Also, $M(\Phi,S)$ consists of a singleton $\{\lambda(t_0)\}$ since the MCL predictor is unique. As $\lim_{m \to \infty} s_m = t_0$, there exists some $m_0 \in \N$ such that the set $M(\Phi_m,S_m)$ consists of a singleton for all $m \geq m_0$ since the MCL predictor is unique in a neighborhood of $t_0$. Clearly, the sets $S_m$, $m \in \N$, are closed convex subsets of $\R^n$ since the function $\Psi:\R^n \to \R$ defined by
\begin{equation*}
\Psi(\lambda) = \int_E \left\vert\sum_{i=1}^n \lambda_i f_{t_i}(x)\right\vert^\alpha m(dx), \quad \lambda=(\lambda_1,\ldots,\lambda_n)^\mathsf{T} \in \R^n,
\end{equation*}
is continuous and convex. It remains to prove that $\lim_{m \to \infty} S_m = S$.

We need to show that $\lim_{m \to \infty} d(\lambda_0,S_m)=0$ for all $\lambda_0 \in S$ and $\liminf_{m \to \infty} d(\lambda_0,S_m)>0$ for all $\lambda_0 \notin S$ according to (\ref{eq:set_convergence}).

First let $\lambda_0 \in S$. We construct a sequence $(\lambda_m)_{m \in \N}$ with $\lambda_m \in S_m$ such that $\lim_{m \to \infty} \lambda_m = \lambda_0$. Let $\varepsilon > 0$. If $\lambda_0 \in S_m$, then we set $\lambda_m:= \lambda_0$. If $\lambda_0 \notin S_m$, then by the definition of $S_m$, it holds $\Psi(\lambda_0) > 0$. Since $\Psi$ is convex, any local minimum is a global one. Obviously $\min_{\lambda \in \R^n} \Psi(\lambda) = \Psi(0)=0$ and therefore, $\lambda_0$ cannot be a local minimum such that
\begin{equation*}
\Psi_{min}:=\min_{\lambda \in B_\varepsilon(\lambda_0)} \Psi(\lambda) < \Psi(\lambda_0) < \max_{\lambda \in B_\varepsilon(\lambda_0)}=:\Psi_{max},
\end{equation*}
where $B_\varepsilon(\lambda_0):=\{\lambda \in \R^n: \Vert \lambda-\lambda_0\Vert_2\leq \varepsilon\}$. As $\lambda_0 \in S$, we have $\Psi(\lambda_0) \leq \kappa(t_0,t_0)$ which implies $\Psi_{min} < \kappa(t_0,t_0)$.  Since $\kappa$ is continuous and $\lim_{m \to \infty} s_m=t_0$, there exists some $m_1 \in \N$ such that $\Psi_{min} < \kappa(s_m,s_m)$ for all $m \geq m_1$. As $\lambda_0 \notin S_m$, we have $\Psi(\lambda_0) > \kappa(s_m,s_m)$ which implies $\kappa(s_m,s_m) < \Psi_{max}$. We have shown that
\begin{equation*}
\Psi_{min} < \kappa(s_m,s_m) < \Psi_{max} \quad \forall m \geq m_1.
\end{equation*}
We can find some $\lambda_m \in B_\varepsilon(\lambda_0)$ by the intermediate value theorem such that $\Psi(\lambda_m) = \kappa(s_m,s_m)$ since $\Psi$ is continuous, that is $\Vert \lambda_m - \lambda_0\Vert_2 \leq \varepsilon$ for all $m \geq m_1$ and $\lim_{m\to\infty}\lambda_m = \lambda_0$.

Now let $\lambda_0 \notin S$. It holds $\Psi(\lambda_0) > \kappa(t_0,t_0)$. Since $\Psi$ is continuous, there exists some $\delta>0$ such that
$$\Psi(\lambda) > \kappa(t_0,t_0) + \frac{\Psi(\lambda_0)-\kappa(t_0,t_0)}{2}, \quad \forall \lambda \in B_\delta(\lambda_0).$$
Due to the continuity of $\kappa$, there exists some $m_2 \in \N$ such that
$$\kappa(s_m,s_m) \leq \kappa(t_0,t_0) + \frac{\Psi(\lambda_0)-\kappa(t_0,t_0)}{2}, \quad \forall m \geq m_2,$$
which implies for all $\lambda \in B_\delta(\lambda_0)$ and $m \geq m_2$
$$\Psi(\lambda) > \kappa(t_0,t_0) + \frac{\Psi(\lambda_0)-\kappa(t_0,t_0)}{2} \geq \kappa(s_m,s_m).$$
Therefore $B_\delta(\lambda_0) \cap S_m = \emptyset$ for all $m \geq m_2$ and thus $\liminf_{m \to \infty} d(\lambda_0,S_m) \geq \delta > 0$.
\end{enumerate}
\end{proof}




\begin{thebibliography}{00}



\bibitem{BC85}
Brockwell, P.~J., Cline, D.~B.~H., Linear prediction of ARMA processes with infinite variance, Stochastic Process. Appl. 19: 281-296 (1985)

\bibitem{BM98}
Brockwell, P.~J., Mitchell, H., Linear prediction for a class of multivariate stable processes, Stoch. Models 14(1): 297-310 (1998)

\bibitem{CS84}
Cambanis, S., Soltani, A.~R., Prediciton of stable Processes: Spectral and moving average representations, Z. Wahrscheinlichkeitstheorie verw. Gebiete 66: 593-612 (1984)

\bibitem{CW92}
Cambanis, S., Wu, W., Multiple Regression on Stable Vectors, J. Multivariate Anal. 41: 243-272 (2004)

\bibitem{CMS76}
Chambers, J.~M., Mallows, C., Stuck, B.W., A method for simulating stable random variables, J. Amer. Statist. Assoc. 71(354): 340-344 (1976)

\bibitem{DeVL93}
DeVore, R.~A., Lorentz, G.~G., Constructive Approximation, Springer-Verlag, Berlin Heidelberg (1993)

\bibitem{GMO00}
Gallardo, J.~R., Matrakis, D., Orozco-Barbosa, L., Prediction of alpha-stable long-range dependent stochastic processes, Adv. Perform. Anal. 3(1): 31-42 (2000)
 
\bibitem{Hil00}
Hill, J.~B., Minimum Dispersion and Unbiasedness: 'Best' Linear Predictors for Stationary ARMA $\alpha$-Stable Processes, Discussion Papers in Economics, Working Paper No. 00-06, Department of Economics, University at Boulder, Colorado (2000), online available at \texttt{http://www.colorado.edu/econ/CEA/papers00/track.pl?p=wp00-6.pdf} (11/30/2010)

\bibitem{Kok96}
Kokoszka, P.~S., Prediction of infinite variance fractional ARIMA: Probab. Math. Statist. 16(1): 65-83 (1996)

\bibitem{Kos91}
Kosmol, P., Optimierung und Approximation, Walter de Gruyter, Berlin (1991)

\bibitem{Kri51}
Krige, D.~G., A statistical approach to some basic mine valuation problems on the Witwatersrand, J. Chem. Metal. Min. Soc. S. Afr. 52(6): 119-139 (1951)

\bibitem{Lan02}
Lantu\'ejoul, C., Geostatistical Simulation: Models and Algorithms, Springer-Verlag, Berlin Heidelberg (2002)

\bibitem{MM09}
Mohammadia M.,  Mohammadpour A., Best linear prediction for $\alpha$-stable random processes,
Statist. Probab. Lett. 79(21): 2266-2272 (2009)

\bibitem{Mol09}
Molchanov, I., Convex and star-shaped sets associated with multivariate stable distributions, I: Moments and densities, J. Multivariate Anal. 100(10): 2195-2213 (2009)
 
\bibitem{Mor78}
Mor\'e, J.~J., The Levenberg-Marquardt algorithm: Implementation and theory, in Watson, G. (ed.): Numerical Analysis, Lecture Notes in Mathematics 60: 105-116, Springer, Berlin Heidelberg (1978)

\bibitem{Pai98}
Painter, S., Numerical method for conditional simulation of L\'evy random fields, Math. Geol. 30(2): 163-179 (1998)

\bibitem{Ros00}
Rosinski, J., Decomposition of stationary $\alpha$-stable random fields, Ann. Probab. 28(4): 1797-1813 (2000)

\bibitem{ST94}
Samorodnitsky, G., Taqqu, M.~S., Stable Non-Gaussian Random Processes, Chapman \& Hall, Boca Raton (1994)

\bibitem{Wac98}
Wackernagel, H., Multivariate Geostatistics, 2nd ed., Springer, Berlin Heidelberg (1998)

\end{thebibliography}



\end{document}